%% file: mainfile.tex
\documentclass[a4paper,12pt,preprint]{article}
\usepackage[english]{babel}
\usepackage[ansinew]{inputenc}
\usepackage{geometry}
\usepackage{color} 
\usepackage{amsmath}
\usepackage{amssymb}
\usepackage{amsthm}
\usepackage{hyperref}
\usepackage{graphicx} 
\usepackage{enumerate}

\newtheorem{theorem}{Theorem}
\newtheorem{proposition}[theorem]{Proposition}
\newtheorem{lemma}[theorem]{Lemma}
\newtheorem{corollary}[theorem]{Corollary}

\theoremstyle{remark}

\newtheorem{remark}{Remark}

\definecolor{orange}{rgb}{1,0.5,0}

\newcommand{\iu}{\mathrm{i}} 

\newcommand{\intd}{\mathrm{d}} 

\newcommand{\lmr}{\text{lmr}} 

\newcommand{\CC}[1]{\mathcal{C}^{#1}} 

\newcommand{\lt}{\underline{t}} 
\newcommand{\gt}{\overline{t}} 
\newcommand{\ltau}{\underline{\tau}} 
\newcommand{\gtau}{\overline{\tau}} 

\newcommand{\quotes}[1]{"#1{}"} 

\newcommand{\lenpar}{0.4\baselineskip} 
\newcommand{\scalefactor}{0.85} 

\newcommand{\C}{\mathbb{C}}
\newcommand{\D}{\mathbb{D}}
\newcommand{\R}{\mathbb{R}}
\newcommand{\N}{\mathbb{N}}

\newcounter{cntselflist}
\newenvironment{selflist}{\begin{list}{\mbox{}\hspace{0.5cm}\mbox{}\arabic{cntselflist})\mbox{ }}{
	\setlength\itemindent{0pt}
	\setlength\leftmargin{0pt}
	\setlength\itemsep{3pt}
	\setlength\labelwidth{0pt}
	\setlength\labelsep{0pt}
	\usecounter{cntselflist}
	}  }{\end{list}}

\begin{document}
	\setcounter{footnote}{1}
	\author{Christoph B{\"o}hm and Wolfgang Lauf}
	\title{A Komatu-Loewner Equation for Multiple Slits}
	\date{\today}
	\maketitle

	%
	\begin{abstract}
		We give a generalization of the Komatu-Loewner equation to multiple slits. 
		Therefore, we consider an $n$-connected circular slit disk $\Omega$ as 
		our initial domain minus $m\in \N$ disjoint, simple and continuous curves that
		grow from the outer boundary $\partial \D$ of $\Omega$ into the interior. Consequently
		we get a decreasing family $(\Omega_t)_{t\in[0,T]}$ of domains with $\Omega_0=\Omega$.
		We will prove that the corresponding Riemann mapping functions $g_t$ from $\Omega_t$ onto
		a circular slit disk, which are normalized by $g_t(0)=0$ and $g_t'(0)>0$, satisfy a 
		Loewner equation known as the Komatu-Loewner equation.
	\end{abstract}

		\input{Input/Chapter1}
		\input{Input/Chapter2}
		\input{Input/Chapter3}
		\input{Input/Chapter4}
		\input{Input/Chapter5}
		\input{Input/Appendix}

	\bibliographystyle{amsplain}

	Christoph B{\"o}hm and Wolfgang Lauf\\
	University of applied sciences Regensburg\\
	93053 Regensburg\\
	Germany\\[0.4\baselineskip]
	\href{mailto:christoph1.boehm@extern.oth-regensburg.de}{christoph1.boehm@extern.oth-regensburg.de}\\
	\href{mailto:wolfgang.lauf@oth-regensburg.de}{wolfgang.lauf@oth-regensburg.de}
\end{document}

%% file: Input/Chapter1.tex
\section{Introduction and results}
Our aim is to generalize the Komatu-Loewner equation (see \cite{BauerFriedrichCSD} or
\cite{Komatu}) to multiple slits. Therefore, we choose circularly slit disks, that is the 
unit disk $\D$ minus proper concentric slits, to be the class of 
standard domains. First of all we start with an $n$-connected ($n\in\N$) circularly slit 
disk $\Omega$ and $m\in\N$ disjoint, simple and continuous curves 
$\gamma_k:[0,T]\rightarrow \bar \Omega$ 
($k=1,\ldots,m$) with strictly increasing parametrizations,
$\gamma_k\big((0,T]\big)\subset \Omega\setminus\{0\}$ and 
$\gamma_k(0)\in\partial\D$. This induces a family 
of conformal mappings $(g_t)_{t\in[0,T]}$, using an extended version of Riemann's well known mapping 
theorem for multiply connected domains:
\[
	g_t: \Omega_t:=\Omega \setminus \left( \bigcup_{k=1}^m \gamma_k\big((0,t]\big) \right) 
	\rightarrow D_t
\]
where $D_t$ denotes a circularly slit disk. If we claim $g_t(0)=0$, $g_t'(0)>0$ and that the 
outer boundary component of $\Omega_t$ corresponds to the unit circle $\partial \D$,
this mapping is unique. See \cite{ConwayII}, Chapter 15.6, for further details. 
Sometimes we write $D(t)$ instead of $D_t$ and $\Omega(t)$ instead of $\Omega_t$.
Later we will show (see Proposition \ref{Pro:Monfunf}) that the function $t\mapsto g'_t(0)$ is 
strictly increasing and continuous in $[0,T]$ (see Proposition \ref{Pro:ConMovFXi}).
As $g_0(z)=z$, it is always possible to choose a parametrization $t(\tau)$ with 
$g'_{t(\tau)}(0)=e^{\tau}$. 
\begin{center}
	\includegraphics[scale=\scalefactor]{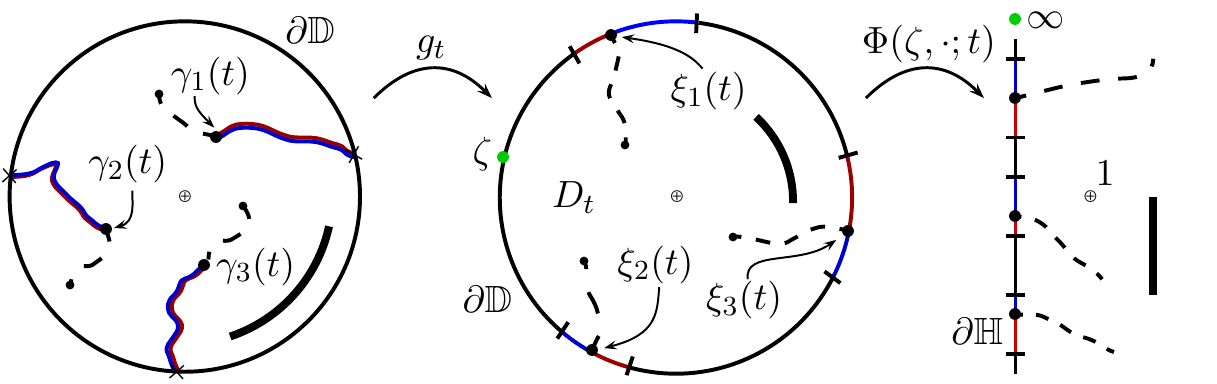} 
\end{center}
A very important quantity in the single slit case and also in
this case is the so called \textit{logarithmic mapping radius} in terms of $\lmr(g):=
\ln|g'(0)|$, where $g$ needs to be analytic and locally injective in $z=0$.
Now, we will define a term that \quotes{measures} the logarithmic mapping radius relative to 
the total logarithmic mapping radius induced by each $\gamma_k$ with $k=1,\ldots,m$. 
Therefore, we introduce the abbreviation
\begin{align*}
	\Gamma_k(t,\tau):= \bigcup_{j=1\atop j\ne k}^m \gamma_j\big( (0,\tau] \big) \cup 
		\gamma_k\big((0,t] \big),
\end{align*}
with $t,\tau\in[0,T]$. We will denote the corresponding normalized mapping function by 
$f_{k;t,\tau}$, i.e.
\[
	f_{k;t,\tau}: \Omega_k(t,\tau):=\Omega\setminus \Gamma_k(t,\tau) \rightarrow D_k(t,\tau)
\]
where $D_k(t,\tau)$ denotes a circularly slit disk. 
\begin{center}
	\includegraphics[scale=\scalefactor]{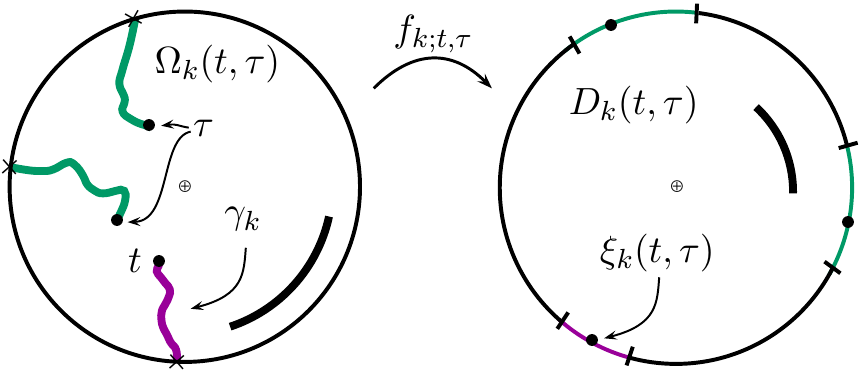} 
\end{center}
Next, we define the following limit
\[
	\lambda_k(t_0) := \lim_{t\rightarrow t_0} \frac{\lmr(f_{k;t,t_0})-
		\lmr(f_{k;t_0,t_0})}{t-t_0},
\]
i.e. $\lambda_k(t_0)$ exists, if $t \mapsto\lmr(f_{k;t,t_0})$ is differentiable at $t_0$.
The first theorem states the existence of the functions $\lambda_k$.
\begin{theorem} \label{The:LamAlmEve}
	Each function $t\mapsto \lambda_k(t)$ (with $k=1,\ldots,m$) exists almost everywhere
	in [0,T]. 
\end{theorem}
The next theorem shows a connection between the existence of $\lambda_k$ and differentiability
of the trajectories $t\mapsto g_t(z)$. Therefore, we denote by $z\mapsto\Phi(\xi,z;t)$ 
for all $t\ge 0$, $\xi\in\partial\D$ and $z\in D_t$ the unique conformal
mapping that maps the circularly slit disk $D_t$ onto the right half-plane minus slits parallel 
to the imaginary axis with $\Phi(\xi,\xi;t)= \infty$ and $\Phi(\xi,0;t)=1$. The existence and
uniqueness of $\Phi$ is given by \cite{Courant}, Theorem 2.3.
\begin{theorem} \label{The:LoKCSDMul}
	Denote by $g_t$ the corresponding mapping function as mentioned before. 
	Assume the values $\lambda_k(t_0)$ ($k=1,\ldots,m$)
	exist for some $t_0\in[0,T]$, then the function $g_t(z)$ is differentiable (w.r.t. $t$) 
	at $t_0$ for every $z\in\Omega_{t_0}$ and it holds
	\[
		\dot g_{t_0}{(z)} = g_{t_0}(z) \sum_{k=1}^m \lambda_k(t_0) \cdot \Phi(\xi_k(t_0),
		g_{t_0}(z);t_0).
	\]
	Furthermore, the functions $\lambda_k(t_0)$ fulfil the condition 
	$\sum_{k=1}^m \lambda_k(t_0) = 1$ if the condition $g_{t}'(0)=c\,e^t$ holds in 
	some neighbourhood of $t_0$ with an arbitrary constant $c>0$.
\end{theorem}
The so called \textit{driving terms} $\xi_k(t_0)$ stands for the images of the tips
$\gamma_k(t_0)$ under the mapping $g_{t_0}$, i.e. $\xi_k(t_0):= g_{t_0}(\gamma_k(t_0))$.
Later (see Proposition \ref{Pro:ConMovFXi}) we will also prove the continuity of the functions
$t\mapsto \xi_k(t)$.\\
If $m=1$, that means we consider the one slit case, we get the following corollary.
\begin{corollary}
	Denote by $g_t$ the corresponding mapping function as mentioned before.
	Assume the slit is parametrized in such a way that $g'_t(0)=e^t$ holds.
	Then the function $g_t(z)$ is differentiable (w.r.t. $t$) 
	for all $t\in[0,T]$ and every $z\in\Omega_{t}$ and it holds
	\[
		\dot g_{t}{(z)} = g_{t}(z)\,  \Phi(\xi_k(t),g_{t}(z);t)
	\]
	for all $t\in[0,T]$ and every $z\in\Omega_{t}$.
\end{corollary}
This result is well known as the \textit{radial one slit Komatu-Loewner equation}, 
see \cite{BauerFriedrichCSD} or \cite{Komatu}.\\
In the simply connected case we get
\[
	\Phi(\zeta, z;t) = \frac{\zeta + z}{\zeta -z},
\]
since $D_t\equiv \D$. If we take this into account, we get the following corollary in case 
of simply connected domains. Here we will also denote by $g_t$ the normalized Riemann map, 
where $\Omega=\D $.
\begin{corollary} \label{Cor:SimpleConCase}
	Denote by $g_t$ the corresponding mapping function as mentioned before. 
	Assume the values $\lambda_k(t_0)$ ($k=1,\ldots,m$)
	exist for some $t_0\in[0,T]$, then the function $g_t(z)$ is differentiable (w.r.t. $t$) 
	at $t_0$ for every $z\in\Omega_{t_0}$ and it holds
	\[
		\dot g_{t_0}{(z)} = g_{t_0}(z) \sum_{k=1}^m \lambda_k(t_0) \cdot \frac{\xi_k(t_0) +
		g_{t_0}(z)}{\xi_k(t_0) - g_{t_0}(z)}.
	\]
	Furthermore, the functions $\lambda_k(t_0)$ fulfil the condition 
	$\sum_{k=1}^m \lambda_k(t_0) = 1$ if the condition $g_{t}'(0)=c\,e^t$ holds in 
	some neighbourhood of $t_0$ with an arbitrary constant $c>0$.
\end{corollary}
In the single slit case, this equation is known as the \textit{slit radial Loewner ODE 
for inclusion chains}.\\
By combining Theorem \ref{The:LamAlmEve} and Theorem \ref{The:LoKCSDMul}
we finally get the following corollary.
\begin{corollary}
	Denote by $g_t$ the corresponding mapping function as in Theorem \ref{The:LoKCSDMul}.
	Then there exists a zero set $\mathcal{N}$ with respect to the Lebesgue measure such that the 
	functions $t\mapsto g_t(z)$ are differentiable on $[0,T]\setminus \mathcal{N}$ 
	for all $z\in \Omega_T$ and it holds
	\[
		\dot g_{t}{(z)} = g_{t}(z) \sum_{k=1}^m \lambda_k(t) \cdot \Phi(\xi_k(t),
		g_{t}(z);t).
	\]
	for all $t\in[0,T]\setminus\mathcal{N}$ and each $z\in\Omega_{t}$.
	Furthermore, the functions $\lambda_k(t_0)$ fulfil the condition 
	$\sum_{k=1}^m \lambda_k(t_0) = 1$ if the condition $g_{t}'(0)=c\,e^t$ holds in 
	some neighbourhood of $t_0$ with an arbitrary constant $c>0$.
\end{corollary}
As a combination of our results with the results of Earle and Epstein from \cite{EarleEpstein}
we immediately get the following result for simply connected domains.
\begin{corollary} \label{Cor:RegSimp}
	Let $n=1$ and let the slits $\gamma_1,\ldots,\gamma_m$ be in $C^2([0,T])$ and regular.
	Denote by $g_t$ the corresponding mapping function as mentioned before.
	Then the limits $\lambda_k(t)$ exist for all $t\in[0,T]$ and it holds for all
	$z\in\Omega_{t}$ and all $t\in[0,T]$
	\[
		\dot g_{t}{(z)} = g_{t}(z) \sum_{k=1}^m \lambda_k(t) \cdot \frac{\xi_k(t) +
		g_{t}(z)}{\xi_k(t) - g_{t}(z)}.
	\]
	Moreover the functions $t\mapsto \lambda_k(t)$ are continuous.
\end{corollary}
A proof of this corollary is given at the end of chapter \ref{Sec:Proof2}.\\[\lenpar]
As an application, see \cite{BoehmSchl}, S. Schlei\ss{}inger and the first author used the methods 
presented in this paper to show that we can always
parametrize the $m$ slits in such a way, that the functions $t\mapsto\lambda_k(t)$ 
($k=1,\ldots,m$) are constant in $[0,T]$, i.e. there are parametrizations of the slits
$\gamma_1,\ldots,\gamma_m$ so that the Komatu-Loewner equation is fulfilled everywhere with
constant coefficients $\lambda_1,\ldots,\lambda_m$.
Hereby Proposition \ref{Pro:ToolConstCoeff} and Theorem \ref{The:LoKCSDMul} play a major role
and we do not assume any regularity for the slits $\gamma_1,\ldots,\gamma_m$.
This generalizes a result of D. Prokhorov, see \cite{Prokhorov:1993}, who considered the simply 
connected case and piecewise analytic slits.
\\[\lenpar]
The remainder of this paper is organized as follows. In section \ref{Sec:ConPro} we will prove
some continuity properties (see Proposition \ref{Pro:ComConf} and Proposition \ref{Pro:ConMovFXi})
and a monotonicity property (see Proposition \ref{Pro:Monfunf}) concerning the 
logarithmic mapping radius. A proof of Theorem \ref{The:LoKCSDMul} will be given in chapter 
\ref{Sec:Proof1}. Chapter \ref{Sec:Proof2} consists primarily of the proof of Theorem 
\ref{The:LamAlmEve} and the proof of a powerful proposition (see Proposition \ref{Pro:DifQuoAbs}), 
which is used in a major way in the proof of the theorem which can be interpreted as the 
connection between Theorem \ref{The:LamAlmEve} and Theorem \ref{The:LoKCSDMul}.
Moreover Proposition \ref{Pro:DifQuoAbs} is 
used in great effect to prove Corollary \ref{Cor:RegSimp} and has a lot of applications in 
\cite{BoehmSchl}.
Finally, in chapter \ref{Sec:GenArb} we generalize Theorem \ref{The:LoKCSDMul} to 
arbitrary finitely connected domains where the curves $\gamma_1,\ldots,\gamma_m$ are 
attached to accessible boundary points.

\begin{remark}[]
One of the first persons who extended Loewner's differential equation to multiply connected domains 
was Komatu \cite{KomatuZweifach}, where he discussed an annulus as the initial domain minus one 
slit that grows from the outer boundary into the interior. Later he derived a result for 
more general circular slit annuli, see \cite{Komatu}. Recently, Bauer and Friedrich continued Komatu's
ideas and found analogous results for circular slit disks and parallel slit domains, see
\cite{BauerFriedrichCSD} and \cite{BauerFriedrichCBC}. All these equations are based on one slit
that grows from the outer boundary component into the interior, where the slit is parametrized 
in such a way that the normalization $g_t'(0)=e^t$ is fulfilled. 
In this work we continue the ideas of Bauer and Friedrich, whereby we have summarized these
in Lemma \ref{Lem:PoiForBF}. 
The proof of this Lemma follows exactly the same steps of \cite{BauerFriedrichCSD}, Theorem 5.1.
The essential difficulty in extending the previously known results to multiple slits is, that
in this case the \quotes{speed factors} $\lambda_k(t_0)$ with respect to the logarithmic capacity  
enter the picture.\\
A proof of the chordal single slit Komatu-Loewner equation mostly based on probabilistic 
arguments is given in \cite{Zhen}.
Furthermore, the authors of \cite{Zhen} point out a problem with the proof of the right 
differentiability of $g_t$ in \cite{BauerFriedrichCSD} and \cite{BauerFriedrichCBC}. 
We consider the left and right differentiability separately.
Our argumentation in the right case is mainly based on the fact, that the image domain 
of $D_t$ under the mapping $w\mapsto \Phi(\xi,w;t)$ is the right half plane minus slits parallel 
to the imaginary axis in combination with an extended version of kernel convergence due to 
Caratheodory, see Lemma \ref{Lem:ConvPhi}.
\end{remark}
\begin{remark}
An interesting question concerning Theorem \ref{The:LoKCSDMul} we do not deal with in this 
paper is the following.
Denote by $\xi_k:[0,T]\rightarrow \partial\D$ and $\lambda_k:[0,T]\rightarrow [0,1]$
for all $k=1,\ldots,m$ continuous functions with $\sum_{k=1}^m \lambda_k\equiv 1$. Moreover
let $\Omega$ be an arbitrary $n$-connected circular slit disk. Consider the implicit initial 
value problem:
\[
	\dot g_{t}{(z)} = g_{t}(z) \sum_{k=1}^m \lambda_k(t) \cdot \Phi(\xi_k(t),
	g_{t}(z);t), \quad
	g_0\equiv\text{id} :\Omega\rightarrow \C
\]
Is there always a (unique) solution to this problem, i.e. a family of normalized 
conformal mappings $g_t: \Omega_t \rightarrow D_t$ for  $t\in [0,T]$ whose driving and weight 
functions equal the given $\zeta_k$ and $\lambda_k$, respectively?
Which regularity conditions on the driving terms $\xi_k$ and weight functions $\lambda_k$
are necessary in order to generate slits?
\end{remark}

%% file: Input/Chapter2.tex
\section{Some basic tools} \label{Sec:ConPro}
The first very useful property concerning the functions $f_{k;t,\tau}$ is given by the following
proposition.
In order to do that, we need an extension of the concept of compact convergence to dynamic domains due to
\cite{ConwayII}, Chapter 15, Definition 4.2. Denote by $(\Omega_n)_{n\in\N}$ a sequence of 
regions with kernel $\Omega$. Furthermore, $0\in\Omega_n$ for all $n\in\N$.
A sequence $f_n:\Omega_n\rightarrow \C$ is said to be locally uniformly convergent on the sequence
$(\Omega_n)_{n\in\N}$ to a function $f:\Omega\rightarrow\C$ if for all $\epsilon>0$ and all 
compact subsets $K$ of $\Omega$ exists a $n_0\in\N$ such that $|f_n(z)-f(z)|<\epsilon$ for all
$z\in K$ and all $n\ge n_0$.
\begin{proposition} \label{Pro:ComConf}
	Let $t_0,\tau_0\in[0,T]$ be fixed and denote by ${(t_n)}_{n\in\N}$ and ${(\tau_n)}_{n\in\N}$
	two sequences (in $[0,T]$) that converge to $t_0$ and $\tau_0$, respectively.
	Then the sequence of functions ${(f_{k,t_n,\tau_n})}_{n\in\N}$ converges locally uniformly to 
	$f_{k;t_0,\tau_0}$ on $(\Omega_k(t_n,\tau_n))_{n\in\N}$.
\end{proposition}
\begin{proof}
	Since there is no risk of confusion, we omit the index $k$.
	Note that $\Omega(t_n,\tau_n)\rightarrow \Omega(t,\tau)$ if $n\rightarrow\infty$, i.e.
	$\Omega(t_0,\tau_0)$ is the kernel of the sequence $(\Omega(t_n,\tau_n))_{n\in\N}$.
	We will prove that the sequence $h_n:=f_{t_n,\tau_n}$ is normal. 
	Let $(h_{n_j})_{j\in\N}$ be a subsequence of $(h_{n})_{n\in\N}$. We show that we find
	a subsequence $(m_j{)}_{j\in\N}$ of $(n_j)_{j\in\N}$ so that for every
	compact set $K\subset \Omega(t_0,\tau_0)$ and for every $\epsilon$ there is a $j_0\in\N$ 
	so that $|h_{m_j}(z)-h_0(z)|<\epsilon$ holds for all $z\in K$ and all $j\ge j_0$.

	Denote by $(K_n)_{n\in\N}$ a sequence of compact sets with $K_n\subset 
	\Omega(t_n,\tau_n)$, $K_n\subset \mathring K_{n+1}$ and 
	$\bigcup_{n\in\N} K_n= \Omega(t_0,\tau_0)$. By Montel's 
	theorem we find a subsequence $(m_{j,1})_{j\in\N}$ of $(n_j)_{j\in\N}$ 
	such that $h_{m_{j,1}}$ converges uniformly
	to an analytic function $h_0$ on $K_1$. Again we find by Montel's theorem a subsequence
	$(m_{j,2})_{j\in\N}$ of $(m_{j,1})_{j\in\N}$ that converges uniformly to $h_0$ on 
	$K_2$. This can be done for every compact set $K_n$, so $h_0$ is an analytic function on
	$\Omega(t_0,\tau_0)$. Thus we define $m_j:=m_{j,j}$, so this subsequence fulfils the 
	condition as stated before.
		
	Next we will show $h_0\equiv f_{t_0,\tau_0}$. By using the extremal property of the 
	functions $h_n$, we can see $h'_n(0)\ge 1$ for all $n\in\N$.
	More details concerning the extremal property of the functions $h_n$ can be found in the
	proof of Proposition \ref{Pro:Monfunf}.
	Thus the function $h_0$ is nonconstant and injective by Hurwitz's theorem. By using an
	extended version of Caratheodory's kernel theorem, we get 
	$D(t_{m_j},\tau_{m_j})\rightarrow\tilde D$, where the limit is in the sense of kernel 
	convergence. Further details can be found in \cite{ConwayII}, Chapter 15.4. 
	Since $\tilde D$ is $n$-connected and represents the kernel of the sequence 
	$\big(D(t_{m_j},\tau_{m_j}\big)_{j\in\N}$, which includes zero, $\tilde D$ is necessarily 
	a circularly slit disk. See \cite{ConwayII}, Chapter 15, Lemma 7.6.
	Summarizing, $h_0$ is a function from $\Omega(t_0,\tau_0)$ onto a circularly slit disk 
	with $h_0(0)=0$ and $h_0'(0)>0$. By using the uniqueness of the extended Riemann mapping 
	theorem we get $h_0\equiv f_{t_0,\tau_0}$.

	As all subsequences $(h_{m_j})_{j\in\N}$ converge to the same function $h_0$ also the whole
	sequence $(h_n)_{n\in\N}$ converges locally uniformly to $h_0$ on 
	$(\Omega(t_n,\tau_n))_{n\in\N}$.
\end{proof}
For later use we define $\xi_k(t,\tau):=f_{k;t,\tau}\big( \gamma_k(t) \big)\in\partial \D$ and 
$\xi_k(t):= \xi_k(t,t)$ as the image of the tip of $\gamma_k([0,t])$ under $f_{k;t,\tau}$,
where $\tau,t\in[0,T]$.
Furthermore, we introduce the abbreviations:
\[
	S_{k;\lt,\gt, \tau}:= f_{k;\lt,\tau}\Big(\gamma_k\big([\lt,\gt] \big)\Big)\subset \D\cup 
		\{\xi_k(\lt,\tau)\},\quad
	s_{k;\lt,\gt, \tau}:= f_{k;\gt,\tau}\Big(\gamma_k\big([\lt,\gt] \big)\Big)\subset\partial\D,
\]
where $0\le\lt \le\gt \le T$ and $\tau,\tau_0 \in [0,T]$.
We note that all the images are with respect to prime ends.
An important fact is given by the following proposition, which describes the movement of the 
functions $\xi_k(t,\tau)$ for fixed $k=1,\ldots,m$. Moreover we will state
an arc convergence of $s_{k;\lt,\gt, \tau}$ and a slit convergence of $S_{k;\lt,\gt, \tau}$ 
in this proposition.
\begin{proposition} \label{Pro:ConMovFXi}
	Let $k\in\{1,\ldots, m\}$ be fixed.
	With the notations from above, the function $(t,\tau)\mapsto \xi_k(t,\tau)$ is continuous in 
	$[0,T]\times[0,T]$. Moreover
	\begin{align*}
		&S_{k;\lt,\gt,\tau} \rightarrow \xi_k(\gt,\tau_0) \text{ as }(\lt,\tau)
			\rightarrow (\gt,\tau_0)\text{ } (\text{where }\lt\nearrow \gt)\\ 
		&s_{k;\lt,\gt,\tau} \rightarrow \xi_k(\lt,\tau_0) \text{ as }(\gt,\tau)
			\rightarrow (\lt,\tau_0)\text{ } (\text{where }\gt\searrow \lt).
	\end{align*}
\end{proposition}
\begin{proof}
	Since there is no risk of confusion, we omit the index $k$.
	We will only show $S_{\lt,\gt,\tau} \rightarrow \xi(\gt,\tau_0)$ as $(\lt,\tau)
	\rightarrow (\gt,\tau_0)$ (where $\lt\nearrow \gt$). The other case $s_{\lt,\gt,\tau}
	\rightarrow \xi(\lt,\tau_0)$ as $(\gt,\tau)\rightarrow (\lt,\tau_0)$ 
	(where $\gt\searrow \lt$) can be proven in the same way.
	Since $\xi(\lt,\tau)\in S_{\lt,\gt,\tau}$ and $\xi(\gt,\tau)\in s_{\lt,\gt,\tau}$, the 
	continuity of $\xi$ follows immediately.
	\begin{center}
		\includegraphics[scale=\scalefactor]{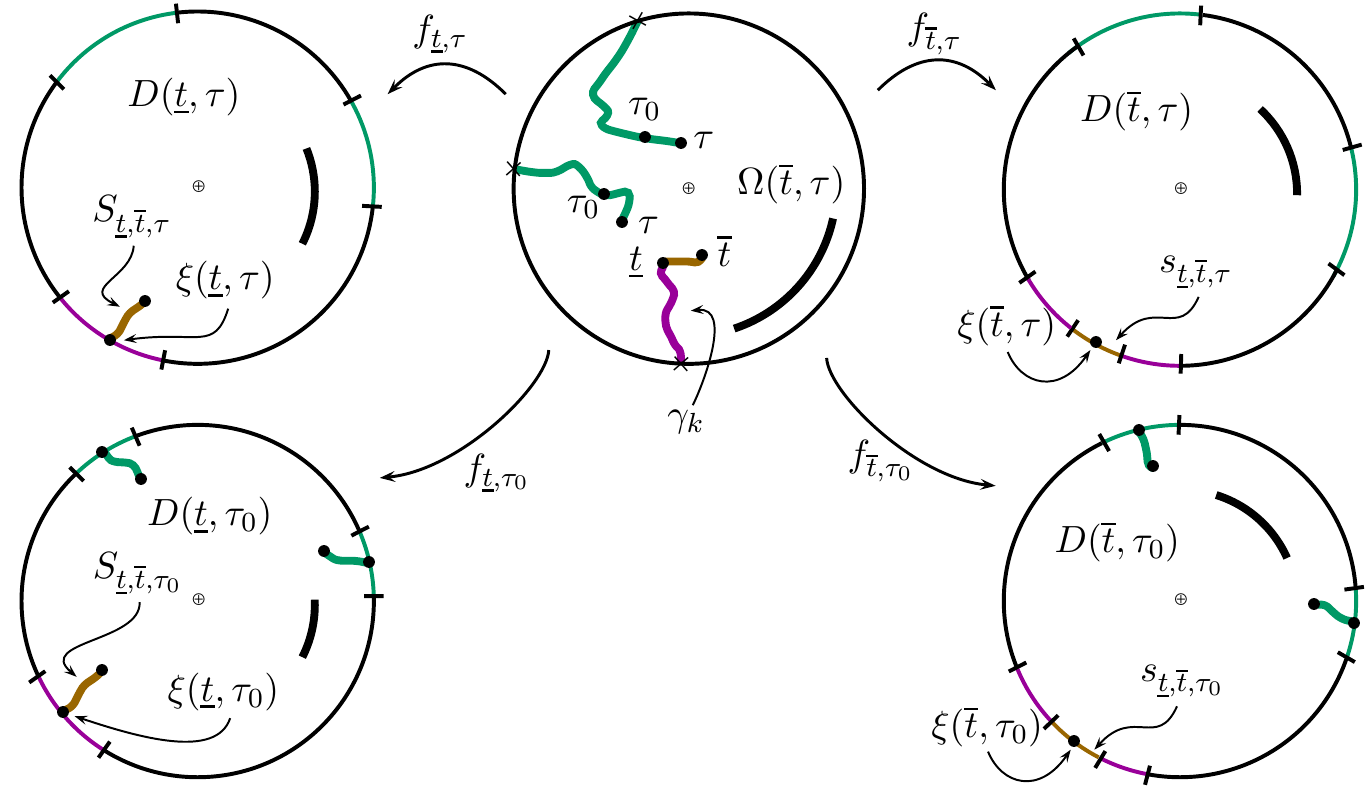} 
	\end{center}
	As described before, we will show that for every $\epsilon >0$ there is a $\delta>0$ with 
	$S_{\lt,\gt,\tau}\subset B_{\epsilon}(\xi(\gt,\tau_0))$ for all $\lt\in[\gt-\delta, \gt]$ and
	$\tau\in[\tau_0-\delta,\tau_0+\delta]$.\\
	Since the function $f_{\gt,\tau_0}$ does not depend on $\lt$ we can can find for every 
	$\epsilon>0$ a $\delta_1>0$ with $s_{\lt,\gt,\tau_0}\subset B_\epsilon(\xi(\gt,\tau_0))$
	for all $\lt\in[\gt-\delta_1,\gt]$. Moreover, the function 
	$f_{\lt,\tau}\circ f_{\gt,\tau_0}^{-1}$ converges by Proposition \ref{Pro:ComConf} 
	locally uniformly to the identity if $(\lt,\tau)$ tends to $(\gt,\tau_0)$. 
	By using the Schwarz reflection principle, we can see that these functions can be extended 
	analytically to $\overline{B_\epsilon(\xi(\gt,\tau_0))}\setminus s_{\lt,\gt,\tau_0}$, 
	if $\epsilon$ and $|\gt-\lt|$ are small enough. 
	Considering the uniform convergence on $\partial B_\epsilon$, we can find a 
	$\delta\in(0,\delta_1)$ with $S_{\lt,\gt,\tau}\subset 
	B_{\epsilon}(\xi(\gt,\tau_0))$ for all $\tau \in[\tau_0-\delta,\tau_0+\delta]$ and all 
	$\lt\in[\gt-\delta,\gt]$.
\end{proof}
Another useful property concerning the logarithmic mapping radius of the mapping $f_{k;t,\tau}$ is 
given by the following proposition.
\begin{proposition} \label{Pro:Monfunf}
	The function $(t,\tau)\mapsto f_{k;t,\tau}'(0)$ is strictly 
	increasing in each variable for fixed $k=1,\ldots,m$.
\end{proposition}
\begin{proof}
	Since there is no risk of confusion, we omit the index $k$.
	First of all we will discuss a very important extremal property concerning the function 
	$f_{t,\tau}$. We denote
	\[
		\mathcal{F}(t,\tau):=\{ f: \Omega(t,\tau)\rightarrow \D\,|\, f\text{ univalent},\, 
		 f(0)=0,\, f'(0)>0\}.
	\]
	Then $f_{t,\tau}$ fulfills the extremal property as follows:
	\[
		f_{t,\tau}'(0)= \max_{f\in\mathcal{F}} f'(0).
	\]
	Further details can be found in \cite{Nehari}, Chapter VII, Section 2.\\[\lenpar]
	Let $0\le\lt<\gt\le T$ and $\tau_0\in[0,T]$. 
	Denote the outer boundary component of the domain $f_{\lt,\tau_0}(\Omega(\gt,\tau_0))$ 
	by $A$. As $\text{int}(A)\varsubsetneq \D$, there is a unique conformal mapping 
	$h: \text{int}(A) \rightarrow \D$ with $h(0)=0$ and $h'(0)>0$. Since $h^{-1}$ fulfills the 
	condition of the Schwarz Lemma, we necessarily get $h'(0)>1$. Thus we have 
	$h\circ f_{\lt,\tau_0}\in \mathcal{F}(\gt,\tau_0)$. This implies 
	\[
		 h'(0) \cdot f'_{\lt,\tau_0}(0) = (h\circ f_{\lt,\tau_0})'(0)  \le f_{\gt,\tau_0}'(0),
	\]
	by using the extremal property. 
	Finally, we have $f'_{\lt,\tau_0}(0) < f_{\gt,\tau_0}'(0)$.\\
	The monotonicity in the second variable $\tau$ can be proven in the same way.
\end{proof}
As an immediate consequence (by setting $\tau=t$) we get the strict monotonicity of the 
function $t\mapsto g_t'(0)$.

%% file: Input/Chapter3.tex
\section{Proof of Theorem \ref{The:LoKCSDMul}}\label{Sec:Proof1}
By adapting the methods in \cite{BauerFriedrichCSD} to multiple slits, we get easily 
\begin{lemma} \label{Lem:PoiForBF}
	Denote by $g_t$ the mapping function as mentioned before. Furthermore, let 
	$0\le \lt< \gt\le T$ and $g_{\gt,\lt}:=g_{\lt}\circ g_{\gt}^{-1}$. Thus we have
	\[
		\log \frac{g_{\gt,\lt}(z)}{z} = \frac{1}{2\pi} \sum_{k=1}^m \int_{\beta_k(\gt,\lt)} 
		\ln\big| g_{\gt,\lt}(\zeta) \big| \cdot \Phi(\zeta,z;\gt) |\intd \zeta|,
	\]
	where $z\mapsto\Phi(\zeta,z; \gt)$ denotes the unique conformal mapping that maps $D(\gt)$ onto 
	the right half plane minus $n-1$ slits parallel to the imaginary axis with the normalization 
	$0\mapsto 1$ and $\zeta\mapsto \infty$. The set $\beta_k(\gt,\lt)$ stands for image of the 
	prime ends of $\gamma_k([\lt,\gt])$ under the mapping $g_{\gt}$.
\end{lemma}
\begin{proof}
	This proof is analogous to the one of Bauer and Friedrich (see \cite{BauerFriedrichCSD}, 
	Theorem 5.1), so most of the details can be found there.
	First of all we consider the function
	\[
		z\mapsto \log \frac{g_{\gt,\lt}(z)}{z}
	\]
	which is analytic in $D(\gt)$. We denote the boundary components of $D(\gt)$ by 
	$C_1(\gt),\ldots, C_n(\gt)$, where $C_n(\gt)=\partial \D$.
	\begin{center}
		\includegraphics[scale=\scalefactor]{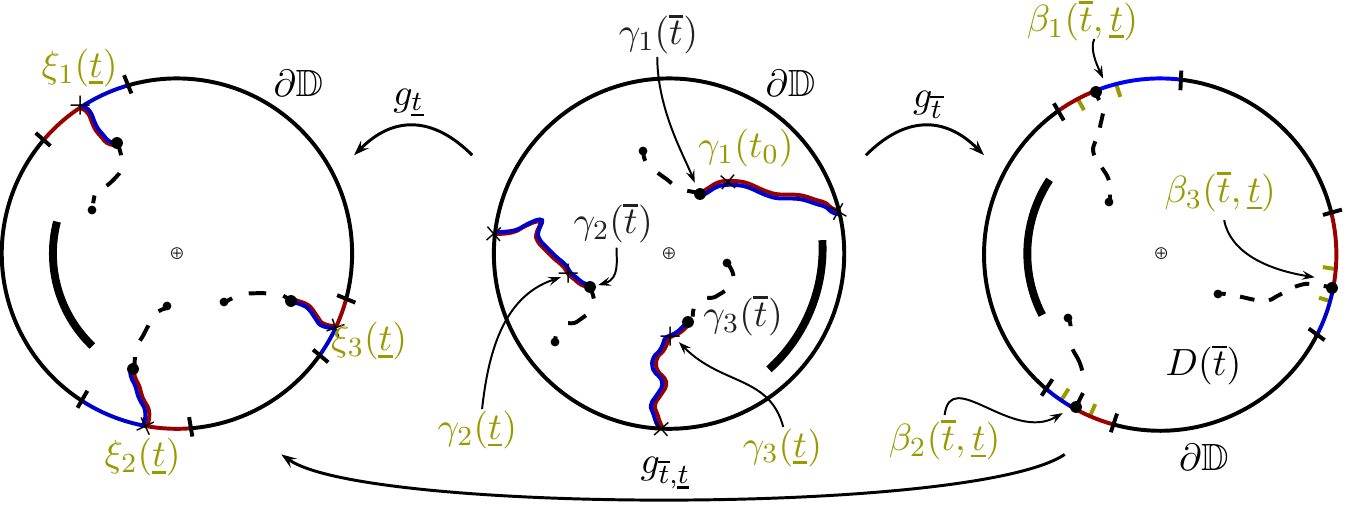}
	\end{center}
	Consequently $\ln|g_{\gt,\lt}(z)/z|$ is harmonic in $D(\gt)$ and by
	Poisson's formula we get
	\[
		\ln\left| \frac{g_{\gt,\lt}(z)}{z} \right| = 
		-\frac{1}{2\pi} \int_{\partial D(\gt)} \ln \left| \frac{g_{\gt,\lt}(\zeta)}{\zeta}\right|
		\frac{\partial G(\zeta,z,\gt)}{\partial n_\zeta} |\intd\zeta|,
	\]
	where $G(\zeta,z;\gt)$ denotes the Green function of $D(\gt)$ with pole at $z$. The left hand 
	side of the previous formula will be denoted as $u(z)$. Since $u$ is harmonic in $D(\gt)$ and 
	the real part of an analytic function, the periods with respect to $C_j(\gt)$ 
	($j=1,\ldots,n-1$) are zero. Thus we get for all $j=1,\ldots, n-1$
	\[ 
		0=\int_{C_j(\gt)} \frac{\partial u}{\partial n_\zeta}(\zeta) |\intd \zeta|
		=\int_{C_j(\gt)} \omega_j(\zeta,\gt) 
			\frac{\partial u}{\partial n_\zeta}(\zeta) |\intd \zeta|
		= \int_{C_j(\gt)} u(\zeta) \frac{\partial \omega_j(\zeta,\gt)}{\partial n_\zeta} 
			|\intd \zeta|,
	\]
	where $\omega_j(\zeta,\gt)$ denotes the harmonic measure with 
	$\omega_j(\zeta,\gt)=\delta_{j,k}$ for $\zeta\in C_k(\gt)$. 
	Note that the last equation is an application of Green's theorem. 
	By combining these two equations we find
	\[
		\ln\left|\frac{g_{\gt,\lt}(z)}{z}\right| = 
		-\frac{1}{2\pi} \int_{\partial D(\gt)} u(\zeta) \left( \frac{\partial G(\zeta,z;\gt)}{
		\partial n_\zeta} + \vec \omega(z,\gt)^T P_{\gt}^{-1} \frac{\partial \vec\omega(\zeta,\gt)}{
		\partial n_\zeta} \right) |\intd \zeta|,
	\]
	where $\vec \omega:= (\omega_1,\ldots,\omega_{n-1})^T$ denotes the harmonic measure vector.
	The matrix $P_{\gt}$ contains the periods of the harmonic measures $w_j$ with respect to 
	$C_k(\gt)$. This matrix is symmetric and positive definite, so the inverse matrix
	exists, see \cite{Nehari}, page 39.
	Furthermore, the function 
	\[
		z\mapsto -\frac{\partial G(\zeta,z;\gt)}{\partial n_\zeta} - \vec \omega(z,\gt)^T 
		P_{\gt}^{-1} \frac{\partial \vec\omega(\zeta,\gt)}{\partial n_\zeta}
	\]
	has a single valued harmonic conjugate. 
	We denote the corresponding analytic function by $\Phi(\zeta,z;\gt)$, where 
	$\Phi(\zeta,z;\gt)$ is determined up to an arbitrary imaginary constant.
	Furthermore, $z\mapsto\Phi(\zeta,z;\gt)$ is an univalent conformal mapping that maps 
	$D(\gt)$ onto the right half plane minus $n-1$ slits parallel to the imaginary axis with
	$\zeta\mapsto\infty$. Details can be found in \cite{Courant}, page 257, or
	\cite{BauerFriedrichCSD}, equation (8).
	Consequently we get
	\[
		\log\left( \frac{g_{\gt,\lt}(z)}{z} \right)=  \frac{1}{2\pi} \int_{\partial D_t}
		\ln \left| \frac{g_{\gt,\lt}(\zeta)}{\zeta} \right| \Phi(\zeta,z;\gt) |\intd \zeta|
		+ \iu c,
	\]
	where $c\in\R$. Since $\ln|g_{\gt,\lt}(\zeta)/\zeta|$ is constant on each
	$C_j(\gt)$ ($j=1,\ldots,n-1$), $C_n(\gt)=\partial \D$ and $\zeta \mapsto G(\zeta,z;\gt) + 
	\vec \omega(z,\gt) P_{\gt}^{-1} \omega(\zeta,\gt)$ has vanishing periods we get
	\[
		\log\left( \frac{g_{\gt,\lt}(z)}{z} \right) = \frac{1}{2\pi}\sum_{k=1}^m 
		\int_{\beta_k(\gt,\lt)} \ln \big| g_{\gt,\lt}(\zeta) \big| 
		\Phi(\zeta,z;\gt) |\intd \zeta| + \iu c,
	\]
	where $\beta_k(\gt,\lt)$ denotes the image of the prime ends of $\gamma_k([\lt,\gt])$ under
	the mapping $g_{\gt}$. Finally, we will show $c=0$. For, we have 
	\[
		0> \ln(g'_{\lt}(0))- \ln(g'_{\gt}(0)) =\log\left(\left. 
		\frac{g_{\gt,\lt}(z)}{z}\right|_{z=0} \right)
	\]
	by normalization of the Riemann mapping function.
	Furthermore, we can choose $\Phi(\zeta,z;\gt)$ in a way that $\Phi(\zeta,0;\gt)=1$ holds.
	If we put both arguments together, we get $c=0$.
\end{proof}
\begin{proof}[Proof for $t\searrow t_0$]
	By using Lemma \ref{Lem:PoiForBF} we get
	\[
		\log \frac{g_{t,t_0}(z)}{z} = \frac{1}{2\pi} \sum_{k=1}^m \int_{\beta_k(t,t_0)} 
		\ln\left| g_{t,t_0}(\zeta) \right| \cdot \Phi(\zeta,z;t) |\intd \zeta|.
	\]
	Since the functions $\zeta\mapsto \ln|g_{t,t_0}(\zeta)|$ and $\zeta\mapsto\Phi(\zeta,z;t)$
	are continuous by Lemma \ref{Lem:ConvPhi} in $\beta_k(t,t_0)$ and $\ln|g_{t,t_0}(\zeta)|\le 0$, we deduce 
	\begin{multline*}
		\log \frac{g_{t,t_0}(z)}{z}=\\ \frac{1}{2\pi} \sum_{k=1}^m \Big(\Re\big(\Phi(
		\zeta_{t,t_0}^{k,1},z;t)\big) + \iu\,\Im\big( \Phi(\zeta_{t,t_0}^{k,2},z;t) \big)\Big)
		\int_{\beta_k(t,t_0)} \ln\left| g_{t,t_0}(\zeta) \right| |\intd \zeta|,
	\end{multline*}
	by using the mean value theorem
	with $\zeta_{t,t_0}^{k,j}\in\beta_k(t,t_0)$ for $j=1,2$.
	Denote the remaining integral by $c_k(t,t_0)$ for every $k=1,\ldots,m$.
	Thus we get by substitution
	\begin{align*}
		c_k(t,t_0) &= 
			\int_{\beta_k(t,t_0)} \ln|g_{t,t_0}(\zeta)|\,|\intd \zeta|\\
		&=\int_{\tilde \beta_k(t,t_0)} \ln|g_{t,t_0}(h_{k;t,t_0}(\zeta))|
			\cdot |h_{k;t,t_0}'(\zeta)|\,|\intd \zeta|\\
		&=\int_{\tilde\beta_k(t,t_0)} \ln|g_{k;t,t_0}(\zeta)|\cdot 
			|h_{k;t,t_0}'(\zeta)|\,|\intd \zeta|.
	\end{align*}
	with $g_{k;t,t_0}:=g_{t_0}\circ f_{k;t,t_0}^{-1}$,
	$h_{k;t,t_0}:= g_{t}\circ f_{k;t,t_0}^{-1}$ and $\tilde\beta_k(t,t_0):=h_{k;t,t_0}^{-1}
	\big(\beta_k(t,t_0)\big)$. We note that $h_{k;t,t_0}$ is analytic in $\tilde\beta_k(t,t_0)$ 
	by the Schwarz reflection principle, so the substitution holds.
	\begin{center}
		\includegraphics[scale=\scalefactor]{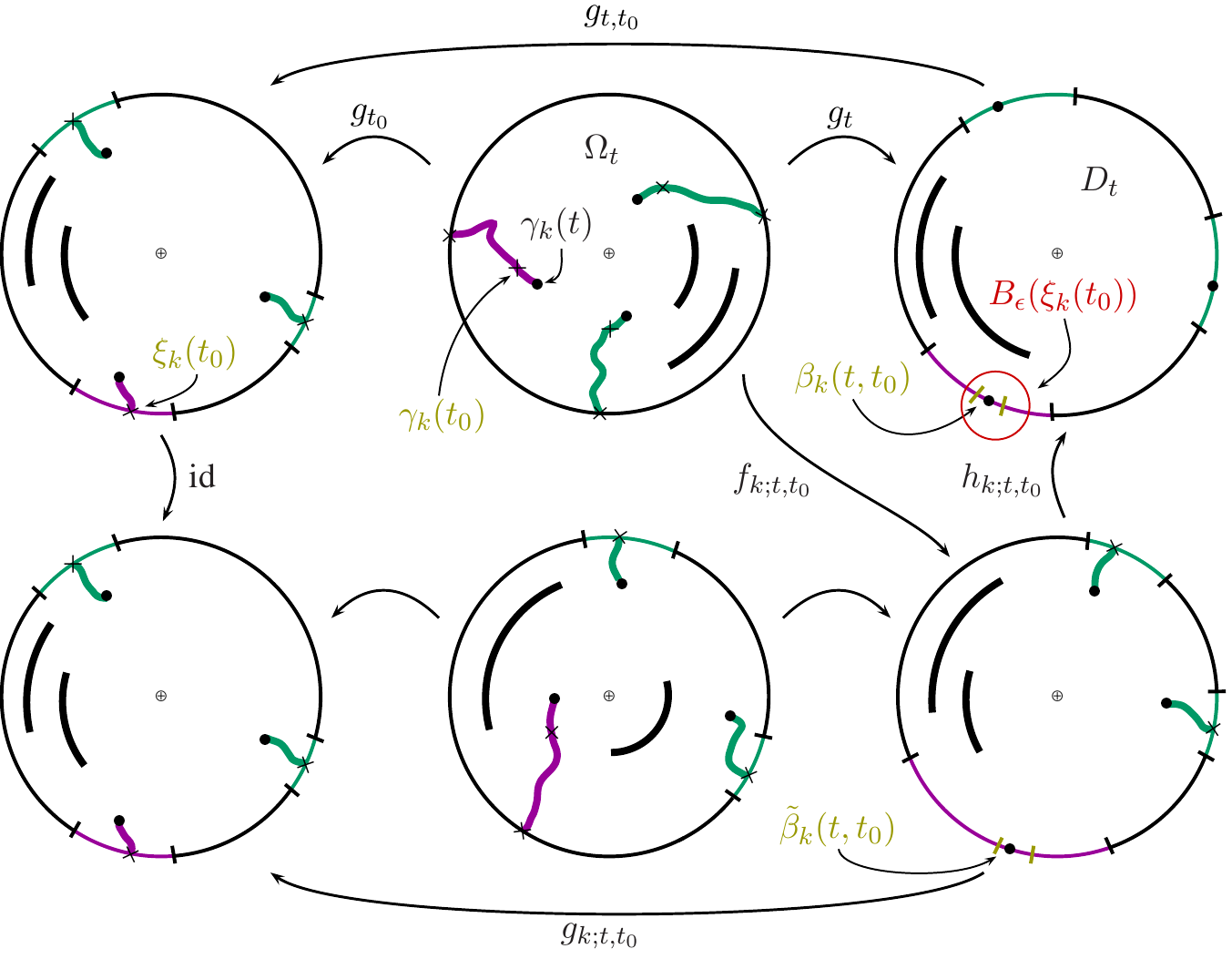}
	\end{center}
	Since the image of the function $g_{k;t,t_0}$ is a subset of $\D$ we have 
	$\ln|g_{k,t,t_0}(\zeta)|\le 0$ for every $\zeta\in \tilde\beta_k(t,t_0)$. 
	Furthermore, the functions
	$\zeta\mapsto h'_{k,t,t_0}(\zeta)$ and $\zeta\mapsto \ln|g_{k;t,t_0}(\zeta)|$ are continuous 
	in $\tilde\beta_k(t,t_0)$, so we get 
	\[
		c_k(t,t_0) = |h'_{k;t,t_0}(\tilde\zeta_{t,t_0})|\int_{\tilde\beta_k(t,t_0)} 
			\ln|g_{k;t,t_0}(\zeta)| \,|\intd \zeta|
	\]
	as an application of the mean-value theorem
	with $\tilde\zeta_{t,t_0}\in \tilde\beta_k(t,t_0)$.
	Furthermore we have $\lim_{t\searrow t_0} h'_{k;t,t_0}(\tilde\zeta_{t,t_0}) =  1$. 
	For by proposition \ref{Pro:ComConf} the functions $h^{-1}_{k;t,t_0}$
	tend to the identity locally uniformly on $D_t$ for $t\searrow t_0$
	and this compact convergence can be extended to a disk $B_\epsilon(\xi_k(t_0))$ for some
	$\epsilon>0$ with
	\[	
		|t-t_0|<\delta \Rightarrow B_\epsilon(\xi_k(t_0))\cap \partial D_t \subset 
		\partial\D
	\]
	(cf. \cite{ConwayII}, Lemma 7.6)
	for a sufficiently small $\delta>0$ by the Schwarz reflection principle.
	By use of Cauchy's theorem and the normalization of $g_{k;t,t_0}$ we see for the real value
	\begin{align*}
		\lmr(g_{k;t,t_0}) &= \log\left(\frac{\intd}{\intd z}\, g_{k;t,t_0}(z) 
				{\Big|}_{z=0}\right)
				= \log\left( \frac{g_{k;t,t_0}(z)}{z} \right){\Big|}_{z=0}\\
			&\hspace{-1.5cm}=\frac{1}{2\pi \iu}\int_{\partial  D_k(t,t_0)} \log\left(\frac{g_{k;t,t_0}(\zeta)}
				{\zeta}\right) \,\frac{\intd \zeta}{\zeta}
			=\frac{1}{2\pi }\int_{\partial D_k(t,t_0)} \log\left(\frac{
				g_{k;t,t_0}(\zeta)}{\zeta}\right) \,\intd \arg \zeta\\
			&\hspace{-1.5cm}= \frac{1}{2\pi }\int_{\partial D_k(t,t_0)} \ln\left|\frac{
				g_{k;t,t_0}(\zeta)}{\zeta}\right| \,\intd \arg \zeta.
	\end{align*}
	The boundary $\partial D_k(t,t_0)$ consists of $\partial\D $ and concentric slits. 
	We note, that the function $\ln\big| \frac{g_{k;t,t_0}(\zeta)}{\zeta}\big|$ is constant 
	on these concentric slits, so we obtain
	\begin{align}\tag{$\star$} \label{Equ:AppCauchy}
		\lmr(g_{k;t,t_0})  =
			\frac{1}{2\pi }\int_{\tilde\beta_k(t,t_0)} 
			\ln\left|\frac{g_{k;t,t_0}(\zeta)}{\zeta}\right| \,|\intd \zeta|
			= \frac{1}{2\pi }\int_{\tilde\beta_k(t,t_0)} 
			\ln|g_{k;t,t_0}(\zeta)| \,|\intd \zeta|.
	\end{align}
	as we integrate on both sides of each concentric slit and $\tilde\beta_k(t,t_0)\subset 
	\partial\D$.
	Since $g_{k;t_0,t_0}\equiv\text{id}$ we have
	\begin{align*}
		 c_k(t,t_0) &= 2\pi \lmr(g_{k;t,t_0})\cdot |h'_{k;t,t_0}(\tilde\zeta_{t,t_0})|\\
		 &=2\pi \big( \lmr(g_{k;t,t_0})-\lmr(g_{k;t_0,t_0})\big) 
		 |h'_{k;t,t_0}(\tilde\zeta_{t,t_0})|.
	\end{align*}
	Now we will use two useful properties concerning the logarithmic mapping radius. These are
	\[
		\lmr (f\circ g) = \lmr(f) + \lmr(g)\quad \text{ and }\quad \lmr(f^{-1})=-\lmr(f),
	\]
	with (in $z=0$) analytic and locally injective functions $g$ and $f$. Thus we get
	\[
		\lmr(g_{k;t,t_0}) = \lmr(f_{k;t,t_0}^{-1}) + \lmr(g_{t_0})
		= -\lmr(f_{k;t,t_0}) + \lmr(g_{t_0})
	\]
	Since $\lmr(g_{t_0})$ does not depend on $t$ it follows
	\[
		\lim_{t\searrow t_0} \frac{c_k(t,t_0)}{t-t_0} = 
		-2\pi \lim_{t\searrow t_0} \frac{\lmr( f_{k;t,t_0})-\lmr(f_{k;t_0,t_0})}{t-t_0}
		=-2\pi \lambda_k(t_0),
	\]
	as $|h'_{k;t,t_0}(\tilde\zeta_{t,t_0})|$ tends to 1 for $t\searrow t_0$. Summarizing we get
	\[
		\lim_{t\searrow t_0} \frac{\log \frac{g_{t_0}(w)}{g_t(w)}}{t-t_0} = \frac{1}{2\pi}
		\lim_{t\searrow t_0}\sum_{k=1}^m \Big[ \Re\big(\Phi(\zeta_{t,t_0}^{k,1},g_t(w);t)\big) + 
		\iu\Im\big(\Phi(\zeta_{t,t_0}^{k,2},g_t(w);t)\big)\Big]
		\frac{c_k(t,t_0)}{t-t_0}.
	\]
	By using Lemma \ref{Lem:ConvPhi} we can see that the functions
	$w\mapsto \Phi(\zeta_{t,t_0}^{k,j},g_t(w),t)$ converge locally uniformly to 
	$\Phi(\xi_k(t_0),g_{t_0}(w),t_0)$ if $t\searrow t_0$.
	Thus we get
	\[
		\lim_{t\searrow t_0} \frac{\log g_{t}(w)-\log g_{t_0}(w)}{t-t_0} = 
		 \sum_{k=1}^m \lambda_k(t_0) \Phi(\xi_k(t_0),g_{t_0}(w),t_0).
	\]
	Finally, by using the same calculation as in (\ref{Equ:AppCauchy}) we see
	\[
		\sum_{k=1}^m \lambda_k(t_0) = -\frac{1}{2\pi}\sum_{k=1}^m \lim_{t\searrow t_0}
		\frac{c_k(t,t_0)}{t-t_0} = - \lim_{t\searrow t_0} 
		\frac{\lmr(g_{t,t_0})}{t-t_0} = 1,
	\]
	where the last equation follows from the condition $g_t'(0)=e^t$ in some neighbourhood of $t_0$
	as assumed in the theorem.
\end{proof}
Before we can prove the other case we need a preliminary lemma.
\begin{lemma}\label{Lem:EstSetA}
	Let $\log\big(\frac{f(z)}{z} \big)$ be analytic on a compact set 
	$A(r_0,\theta_1,\theta_2)=\{z=r e^{\iu \theta}\in\C\,|\, r\in[r_0,1],\, \theta 
	\in[\theta_1,\theta_2] \}$, where $0<r_0<1$ and $|\theta_2-\theta_1|< 2\pi$. Furthermore,
	$|f(z)|=1$ for all $z\in\partial \D \cap A(r_0,\theta_1,\theta_2)$ and for some $\delta>0$
	the inequality
	\[
		\left|\frac{\intd }{\intd z} \log\left(\frac{f(z)}{z}\right)\right|<
		\delta
	\]
	holds for all $z\in A(r_0,\theta_1,\theta_2)$. Then the inequality 
	\[
		{|z|}^{1+\delta} \le |f(z)| \le {|z|}^{1-\delta}
	\]
	is true for all $z\in A(r_0,\theta_1,\theta_2)$.
\end{lemma}
\begin{proof}
	Let $\gamma_\theta(r):=r\cdot e^{\iu \theta}$ with $r\in [r_0,1]$ and 
	$\theta\in[\theta_1,\theta_2]$. Denote by
	\[
		h_{\theta}(r):= \Re\left(\log \left(\frac{f
		(\gamma_\theta(r))}{\gamma_\theta(r)} \right) \right)=
		\ln \left|\frac{f(\gamma_\theta(r))}{\gamma_\theta(r)} \right|.
	\]
	Consequently the function $h_{\theta}$ is in $\CC{\infty}((r_0,1),\R)$. Thus we get
	\begin{align*}
		\left| \frac{\partial}{\partial r}{h}_\theta(r)\right| &= 
		\left|\Re\left(\left. {\frac{\intd}{\intd z} 
		\log\left( \frac{f(z)}{z} \right)}\right|_{z=\gamma_\theta(r)}
		\cdot \dot\gamma_\theta( r) \right)\right|\\
		&= \left|\Re\left(\left. {\frac{\intd}{\intd z} 
		\log\left( \frac{f(z)}{z} \right)}\right|_{z=\gamma_\theta(r)}
		\cdot e^{\iu \theta} \right)\right|\\
		&\le \left|\left. {\frac{\intd}{\intd z} 
		\log \left(\frac{f(z)}{z} \right)}\right|_{z=\gamma_\theta(r)}
		\right| \le \delta.
	\end{align*}
	Furthermore, we have $h_{\theta}(1)=0$, $\pm\delta \ln(1) = 0$ and $\delta \frac{1}{r}\ge 
	\delta$ (for all $r\in[r_0,1]$), so we get 
	\[
		\ln(r^\delta) =\delta \ln(r) \le h_{\theta}(r) \le -\delta \ln(r) = \ln(r^{-\delta}),
	\]
	for all $r\in[r_0,1]$ and $\theta\in[\theta_1,\theta_2]$.
	Since to every $z\in A(r_0,\theta_1,\theta_2)$ corresponds a $r\in[r_0,1]$ and a 
	$\theta\in[\theta_1,\theta_2]$ with $\gamma_\theta(r)=z$, we get for given 
	$z\in A(r_0,\theta_1,\theta_2)$
	\[
		\ln(|z|^{\delta}) \le \ln\left| \frac{f(z)}{z} \right| \le \ln(|z|^{-\delta}).
	\]
	Finally, we get the asserted equation by applying the exponential function.
\end{proof}
According to this lemma we define $A_\epsilon(\zeta):= A(1-\epsilon, \theta-\pi\epsilon, 
\theta+ \pi\epsilon)$, where $\zeta= e^{\iu \theta}$ and $0<\epsilon<1$.
\begin{lemma} \label{Lem:SetA}
	Let be $(\tau_n)_{n\in\N}$ a sequence (in $[0,T]$) converging to $\tau_0$ and 
	$(\lt_n)_{n\in\N}$ and $(\gt_n)_{n\in\N}$ converging to $t_0$, where 
	$0\le\lt_n\le \gt_n\le T$ holds for all $n\in\N$. Then we can find for every $\epsilon >0$ 
	a $n_0\in\N$ such that
	\[
		f_{k; \lt_n, \tau_n}\Big( \gamma_k\big([\lt_n, \gt_n] \big) \Big) \subset 
		A_\epsilon \big(\xi_k(t_0,\tau_0) \big)
	\]
	holds for all $n\ge n_0$.
\end{lemma}
\begin{proof}
	This lemma follows immediately from Proposition \ref{Pro:ConMovFXi}, as the diameter of the set 
	$f_{k; \lt_n, \tau_n}\big( \gamma_k([\lt_n, \gt_n] )\big)$ tends to zero 
	and $\xi_k(t_n,\tau_n)$ tends to $\xi_k(t_0,\tau_0)$  
	if $n\rightarrow \infty$ and $\xi_k(\lt_n,\tau_n)\in f_{k; \lt_n, \tau_n}\big( 
	\gamma_k([\lt_n, \gt_n] ) \big)$ for all $n\in\N$.
\end{proof}
\begin{proof}[Proof of Theorem \ref{The:LoKCSDMul} in the case $t\nearrow t_0$]
	Analog to the other case we start with
	\[
		\log \frac{g_{t_0,t(z)}}{z} = 
		\frac{1}{2\pi} \sum_{k=1}^m \Big(\Re\big(\Phi(\zeta_{t_0,t}^{k,1},z;t_0)\big)+
		\iu\Im\big(\Phi(\zeta_{t_0,t}^{k,2},z;t_0)\big)\Big) \int_{\beta_k(t_0,t)} 
		\ln|g_{t_0,t}(\zeta)| \, |\intd\zeta|
	\]
	where $\zeta_{t_0,t}^{k,j}\in\beta_{k}(t_0,t)$ with $j=1,2$. We denote the integral
	by $c_k$, so we get
	\[
		c_k(t_0,t) = \int_{\beta_k(t_0,t)} \ln|g_{t_0,t}(\zeta)|\,|\intd \zeta| =
		\int_{\beta_k(t_0,t)} \ln|h^{-1}_{k;t_0,t}(g_{k;t_0,t}(\zeta))|\, |\intd \zeta|.
	\]
	with the abbreviations $g_{k;t_0,t}:=f_{k;t,t_0}\circ g_{t_0}^{-1}$ and 
	$h_{k;t_0,t}:=f_{k;t,t_0}\circ g_{t}^{-1}$.
	\begin{center}
		\includegraphics[scale=\scalefactor]{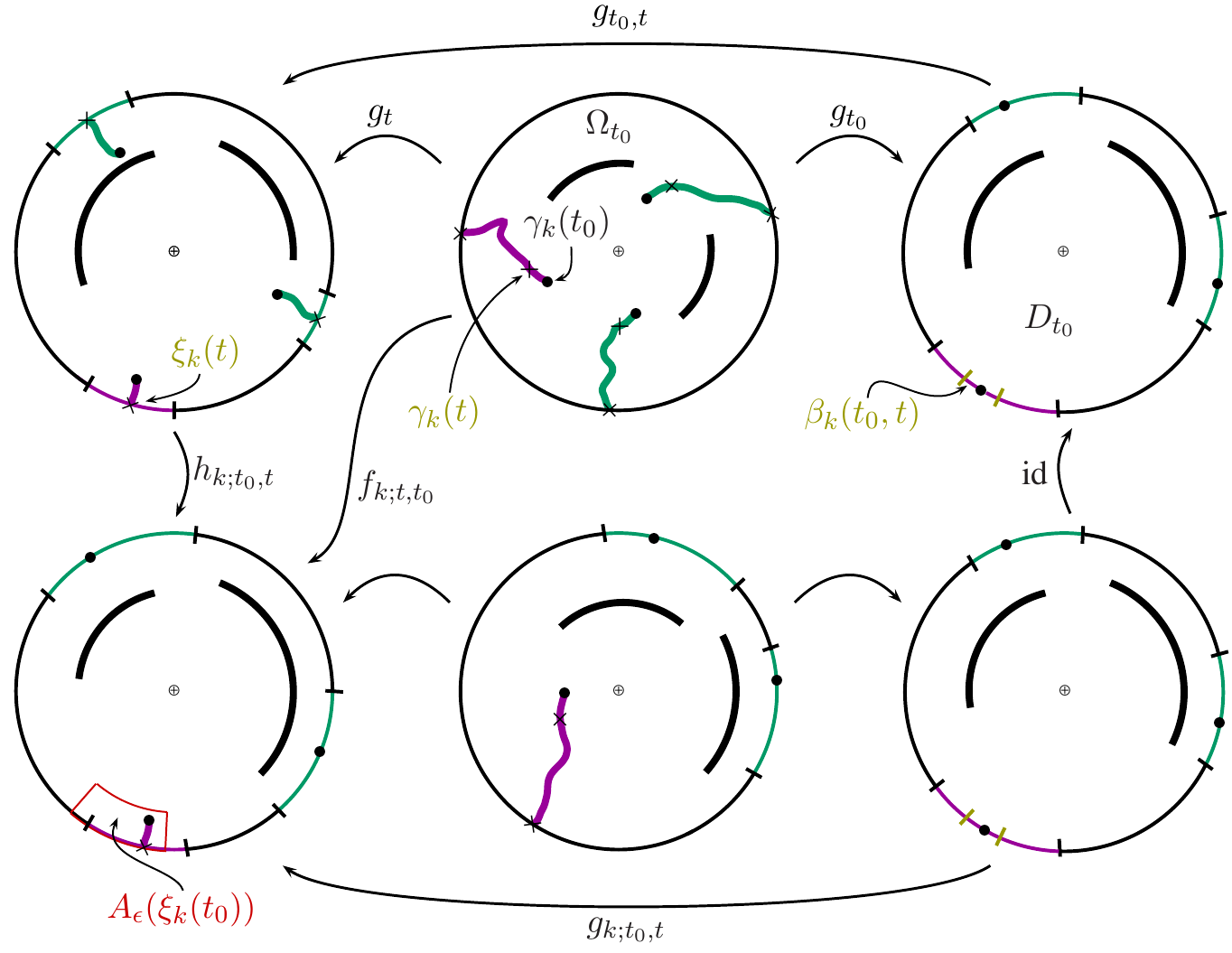}
	\end{center}
	We will now discuss the map $h^{-1}_{k;t_0,t}$. 
	First of all, we notice that (by Lemma \ref{Lem:SetA}) for each $\epsilon>0$ there is a
	$\rho_1(\epsilon)>0$ such that the images of the set $\gamma\big([t,t_0] \big)$ under 
	the functions $f_{k;t,t_0}$ lie in $A_\epsilon\big(\xi_k(t_0)\big)$ for all 
	$t\in[t_0-\rho_1,t_0]$.
	If $\epsilon>0$ is chosen small enough, the functions $h^{-1}_{k;t_0,t}$ are 
	analytic in $A_\epsilon\big(\xi_k(t_0)\big)$ by reflection.
	
	Consequently, the functions $\log\big(h^{-1}_{k;t_0,t}(z)/z\big)$ are analytic in 
	$A_\epsilon\big(\xi_k(t_0)\big)$. Furthermore, this sequence of functions tends (by Proposition 
	\ref{Pro:ComConf}) locally uniformly on $D_k(t,t_0)$ for $t\nearrow t_0$ to the 
	function identically zero as $t$ tends to $t_0$, so the first derivative tends 
	(uniformly on compact sets) to zero too. 
	Thus we find for every $\delta>0$ a $\rho>0$ with
	\[
		\left|\frac{\intd}{\intd z} \log\left(\frac{h^{-1}_{k;t_0,t}(z)}{z} \right) \right| < 
		\delta
	\]
	for all $z\in A_\epsilon\big(\xi_k(t_0)\big)$ and $t\in [t_0-\rho,t_0]$, where $\rho<\rho_1$.
	Hence we can use Lemma \ref{Lem:EstSetA} to get
	\[
		|z|^{1+\delta} \le |h^{-1}_{k,t_0,t}(z)| \le |z|^{1-\delta}
	\]
	for all $z\in A_\epsilon\big(\xi_k(t_0)\big)$ and $t\in [t_0-\rho,t_0]$.
	By using this estimate we get
	\begin{multline*}
		{(1+\delta_{t_0,t})} \int_{\beta_k(t_0,t)} \ln{|g_{k;t_0,t}(\zeta)|} |\intd\zeta|\\
		\le \int_{\beta_k(t_0,t)} \ln{|h^{-1}_{k;t_0,t}(g_{k;t_0,t}(\zeta)|} |\intd\zeta|
		\le {(1-\delta_{t_0,t})} \int_{\beta_k(t_0,t)} \ln{|g_{k;t_0,t}(\zeta)|} 
		|\intd\zeta|,
	\end{multline*}
	where $\delta_{t_0,t}$ is chosen in such a way, that the previous estimation holds and 
	$\delta_{t_0,t}\rightarrow 0$ as $t\nearrow t_0$. 
	Furthermore, we find 
	\[
		\int_{\beta_k(t_0,t)} \ln|g_{k;t_0,t}(\zeta)| |\intd\zeta|=
		2\pi \big(\lmr(g_{k;t_0,t})-\lmr(g_{k;t_0,t_0}) \big)
 	\]
	in the same way as in equation (\ref{Equ:AppCauchy}).
	Thus we get 
	\[
	 	\lim_{t\nearrow t_0} \frac{c_k(t_0,t)}{t-t_0} = 
		2\pi\lim_{t\nearrow t_0} \frac{\lmr(f_{k;t,t_0})-\lmr(f_{k;t_0,t_0})}{t-t_0}
		=2\pi \lambda_k(t_0).
	\]
	by using $\lmr(f_{k;t,t_0})=\lmr(g_{k;t,t_0})+\lmr(g_{t_0})$ combined with the
	previous inequality.
	Summarizing we find
	\begin{align*}
		\lim_{t\nearrow t_0} &\frac{\log \frac{g_t(w)}{g_{t_0}(w)}}{t-t_0} \\ 
		&=\frac{1}{2\pi}
		\sum_{k=1}^m \lim_{t\nearrow t_0}\Big[ \Re\big(\Phi(\zeta_{t_0,t}^{k,1},g_{t_0}(w);
		t_0)\big) + \iu\Im\big(\Phi(\zeta_{t_0,t}^{k,2},g_{t_0}(w);t_0)\big)\Big]\cdot 
		\lim_{t\nearrow t_0}\frac{c_k(t,t_0)}{t-t_0}\\
		&=\sum_{k=1}^m \lambda_k(t_0) \cdot \Phi(\xi_k(t_0), g_{t_0}(w);t_0),
	\end{align*}
	so the differential equation is proved.
\end{proof}

%% file: Input/Chapter4.tex
\section{Proof of Theorem \ref{The:LamAlmEve}}\label{Sec:Proof2}
First of all we will introduce the abbreviations
\[
	T_{k;t,\ltau,\gtau} := f_{k;t,\gtau} \circ f_{k;t,\ltau}^{-1},\quad
	u_{k;\lt,\gt,\tau} := f_{k;\lt,\tau} \circ f_{k;\gt,\tau}^{-1}.
\]
for all $0\le \ltau \le \gtau<T$ and $0\le t \le T$.
Before we can prove Theorem \ref{The:LamAlmEve} we need some lemmas.
\begin{lemma} \label{Lem:FunOneAbs}
	For each $\epsilon>0$ there exists a $\delta>0$ so that for all $0\le\lt,\gt,\ltau,\gtau\le T$
	with $0\le\gt-\lt, \gtau-\ltau<\delta$ and all $z\in s_{k;\lt,\gt,\ltau}$ holds
	\[
		\big| T_{k;\gt,\ltau,\gtau}'(z) -1 \big|<\epsilon.
	\]
\end{lemma}
\begin{center}
	\includegraphics[scale=\scalefactor]{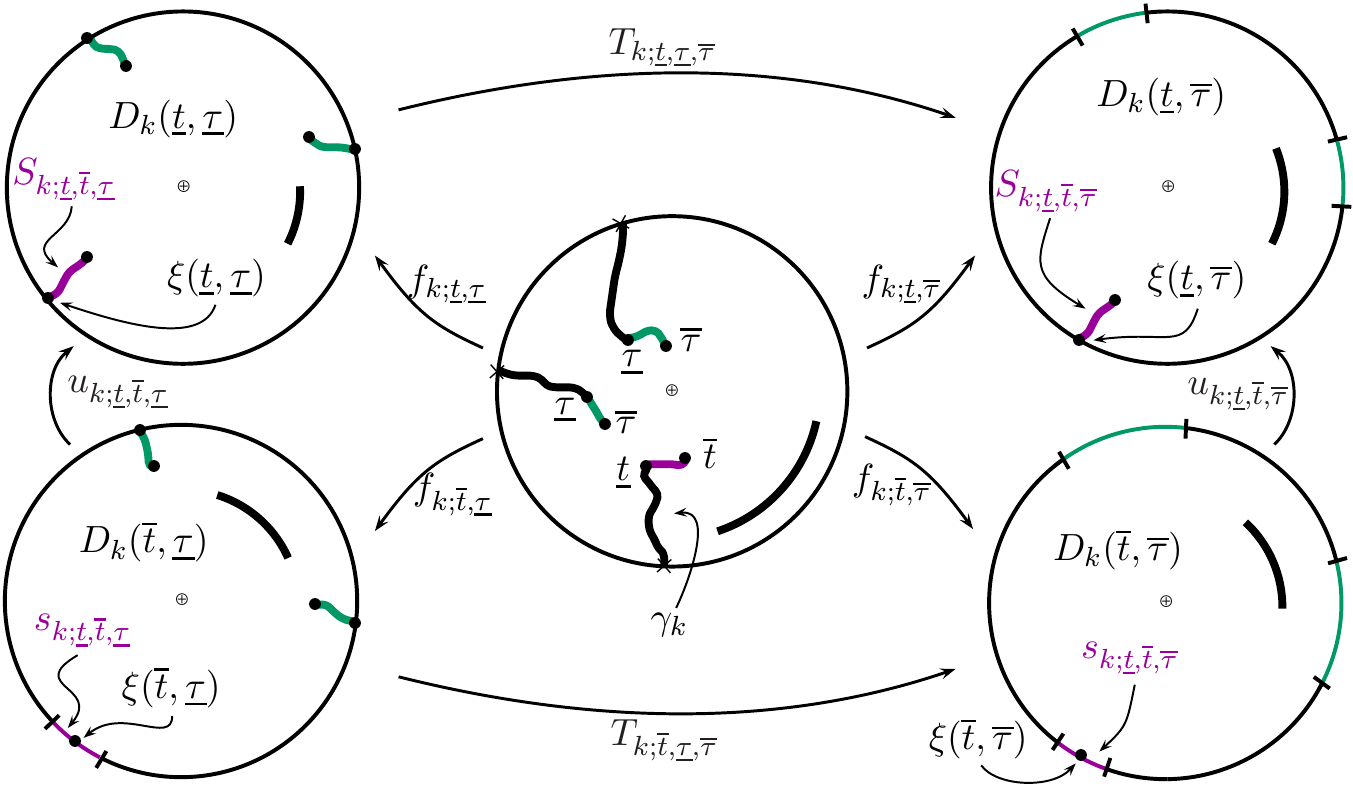} 
\end{center}
\begin{proof}
	Since there is no risk of confusion, we omit the index $k$.
	We assume the opposite:
	Assume there exists an $\epsilon > 0$ so that for all $\delta>0$ exists 
	$0\le\lt,\gt,\ltau,\gtau\le T$
	with $0\le\gt-\lt,\gtau-\ltau<\delta$ and $z\in s_{\lt,\gt,\ltau}$ 
	so that 
	\[
		\big| T_{\gt,\ltau,\gtau}'(z) -1 \big|\ge\epsilon
	\]
	holds. Let ${(\delta_n)}_{n\in\N}$ be a sequence that converges to zero, and denote by
	$(\lt_n)_{n\in\N}$, $(\gt_n)_{n\in\N}$, $(\ltau_n)_{n\in\N}$, $(\gtau_n)_{n\in\N}$ and
	$z_n\in s_{\lt_n,\gt_n,\ltau_n}$ sequences that fulfil the 
	condition described before. Without restricting 
	generality we can assume (by boundedness) that all these sequences are convergent, i.e.
	\[
		\lt_n \rightarrow t_0 \leftarrow \gt_n,\quad \ltau_n\rightarrow \tau_0 \leftarrow \gtau_n,
		\quad z_n\rightarrow z_0=\xi(t_0,\tau_0),
	\]
	where the last equation follows from Proposition \ref{Pro:ConMovFXi}.
	Thus we find for every $\rho>0$ a $n_0\in\N$ with $z_n\in B_\rho(z_0)$ for all $n\ge n_0$.
	If $\rho$ is small enough we find a $n_0^*\ge n_0$ so that the functions 
	$T_{\gt_n,\ltau_n,\gtau_n}$ are analytic in $B_\rho(z_0)$ by reflection.
	Consequently by using Proposition \ref{Pro:ComConf} the sequence of functions 
	$T_{\gt_n,\ltau_n,\gtau_n}$ converges
	uniformly on $B_\rho(z_0)$ to the identity mapping. Thus $T_{\gt_n,\ltau_n,\gtau_n}'$ 
	converges uniformly on $B_\rho(z_0)$ to 1. This is 
	a contradiction, so the proof is complete.
\end{proof}
\begin{lemma} \label{Lem:LogFunAbs}
	For all $\epsilon>0$ exists a $\delta>0$ so that for all 
	$0\le\lt,\gt,\ltau,\gtau\le T$ with $0\le\gt-\lt, \gtau-\ltau<\delta$ 
	exists a $\rho>0$ with $S_{k;\lt,\gt,\gtau} \subset 
	A_\rho(\xi_k(\lt,\gtau))$ so that for all $z \in A_\rho(\xi_k(\lt,\gtau))$ holds
	\[
		\left| \frac{\intd}{\intd z} \log\left( \frac{T_{k;\lt,\ltau,\gtau}^{-1}(z)}
		{z} \right) \right|<\epsilon.
	\]
\end{lemma}
If we combine this lemma with Lemma \ref{Lem:EstSetA} we can easily derive the following lemma.
\begin{lemma} \label{Lem:FunTwoAbs}
	For all $\epsilon>0$ exists a $\delta>0$ so that for all 
	$0\le\lt,\gt,\ltau,\gtau\le T$ with $0\le\gt-\lt, \ltau-\gtau<\delta$ 
	and $z\in S_{k;\lt,\gt,\gtau}$ holds
	\[
		{|z|}^{1+\epsilon} \le |T_{k;\lt,\ltau,\gtau}^{-1}(z)| \le
		{|z|}^{1-\epsilon}
	\]
\end{lemma}
\begin{proof}[Proof of Lemma \ref{Lem:LogFunAbs}]
	Since there is no risk of confusion, we omit the index $k$.
	We assume the opposite:
	Assume there exists an $\epsilon > 0$ so that for all $\delta>0$ exists 
	$0\le\lt,\gt,\ltau,\gtau\le T$
	with $0\le\gt-\lt, \gtau-\ltau<\delta$ so that for all 
	$\rho>0$ with $S_{\lt,\gt,\gtau} \subset A_\rho(\xi(\lt,\gtau))$
	exists a $z \in A_\rho(\xi(\lt,\gtau))$ so that 
	\[
		\left| \frac{\intd}{\intd z} \log\left( \frac{T_{\lt,\ltau,\gtau}^{-1}(z)}
		{z} \right) \right| \ge \epsilon
	\]
	holds. Let ${(\delta_n)}_{n\in\N}$ be a sequence that converges to zero, and denote by
	$(\lt_n)_{n\in\N}$, $(\gt_n)_{n\in\N}$, $(\ltau_n)_{n\in\N}$ and $(\gtau_n)_{n\in\N}$ 
	sequences as described before. Without restricting
	generality we can assume that all these sequences are convergent, i.e.
	\[
		\lt_n \rightarrow t_0 \leftarrow \gt_n,\quad \ltau_n\rightarrow \tau_0 \leftarrow \gtau_n.
	\]
	As $n$ tends to infinity, $\xi(\lt_n,\gtau_n)$ tends to $\xi(t_0,\tau_0)$ and the 
	diameter of the set $S_{\lt_n,\gt_n,\gtau_n}$ tends to zero. Thus
	we can find a sequence $(\rho_n)_{n\in\N}$ that fulfils the condition 
	$S_{\lt_n,\gt_n,\gtau_n} \subset A_{\rho_n}(\xi(\lt_n,\gtau_n))$
	and tends to zero. Consequently we find for every given $\psi>0$ a $n_0\in\N$ with
	$A_{\rho_n}(\xi(\lt_n,\gtau_n)) \subset A_\psi(\xi(t_0,\tau_0))$ for every $n\ge n_0$.
	If $\psi$ is small enough, we find a $n_0^*\ge n_0$ so that the functions 
	$T_{\lt_n,\ltau_n,\gtau_n}^{-1}$ are analytic in $A_\psi(\xi(t_0,\tau_0))$
	by reflection for all $n\ge n_0^*$. Consequently the sequence of functions 
	$T_{\lt_n,\ltau_n,\gtau_n}^{-1}$ converges uniformly on $A_\psi(\xi(t_0,\tau_0))$ 
	to the identity mapping. Thus the sequence of functions
	\[
		\left| \frac{\intd}{\intd z} \log\left( \frac{(T_{\lt_n,\ltau_n,\gtau_n}^{-1}(z)}{z}
		\right) \right| 
	\]
	converges there uniformly to zero. This is a contradiction, so the proof is complete.
\end{proof}
\begin{proposition} \label{Pro:DifQuoAbs}
	For all $\epsilon>0$ exists a $\delta>0$ so that for all 
	$0\le\lt,\gt,\ltau,\gtau\le T$ with $0<\gt-\lt<\delta$ and $0\le \gtau-\ltau<\delta$
	holds
	\[
		1-\epsilon <\frac{\lmr(f_{k;\lt,\ltau})-\lmr(f_{k;\gt,\ltau})}{\lmr(f_{k;\lt,\gtau})-
		\lmr(f_{k;\gt,\gtau})} < 1+\epsilon.
	\]
\end{proposition}
\begin{proof}
	Since there is no risk of confusion, we omit the index $k$.
	As $0\le\lt,\gt,\ltau,\gtau\le T$ with $0<\gt-\lt$ and $0\le\gtau-\ltau$ we get 
	by using Cauchy's formula and the mean value theorem of integration
	\begin{align*}
		\frac{\lmr(f_{\lt,\ltau})-\lmr(f_{\gt,\ltau})}{\lmr(f_{\lt,\gtau})-
				\lmr(f_{\gt,\gtau})}
			&=\frac{\ln\big(u'_{\lt,\gt,\ltau}(0)\big)}{\ln\big(u'_{\lt,\gt,\gtau}(0)\big)}
				=\frac{ \int_{s_{\lt,\gt,\ltau}} \ln|u_{\lt,\gt,\ltau}(z)| \,|\intd z|}
				{ \int_{s_{\lt,\gt,\gtau}} \ln|u_{\lt,\gt,\gtau}(z)| \,|\intd z|}\\
			&\hspace{-3.5cm}=\frac{ \int_{s_{\lt,\gt,\ltau}} \ln\big|T^{-1}_{\lt,\ltau,\gtau}
				\big(u_{\lt,\gt,\gtau}(T_{\gt,\ltau,\gtau}(z))\big)\big| \,|\intd z|}
				{\int_{s_{\lt,\gt,\gtau}} \ln|u_{\lt,\gt,\gtau}(z)| \,|\intd z|}
			=\frac{ \int_{s_{\lt,\gt,\ltau}} \ln\big|T^{-1}_{\lt,\ltau,\gtau}
				\big(u_{\lt,\gt,\gtau}(T_{\gt,\ltau,\gtau}(z))\big)\big| \,|\intd z|}
				{\int_{s_{\lt,\gt,\ltau}} \ln\big|u_{\lt,\gt,\gtau}\big(T_{\gt,\ltau,\gtau}(z)
				\big)\big| \cdot |T'_{\gt,\ltau,\gtau}(z)|  \,|\intd z|}\\
			&\hspace{-3.5cm}=\frac{ \int_{s_{\lt,\gt,\ltau}} \ln\big|T^{-1}_{\lt,\ltau,\gtau}
				\big(u_{\lt,\gt,\gtau}(T_{\gt,\ltau,\gtau}(z))\big)\big| \,|\intd z|}
				{|T'_{\gt,\ltau,\gtau}(z^*)| \int_{s_{\lt,\gt,\ltau}} \ln\big|u_{\lt,\gt,\gtau}
				\big(T_{\gt,\ltau,\gtau}(z) \big)\big| \,|\intd z|}.
	\end{align*}
	Consequently we find by Lemma \ref{Lem:FunTwoAbs}
	for any given $\epsilon>0$ a $\delta>0$ with
	\[
		\frac{1-\epsilon}{|T'_{\gt,\ltau,\gtau}(z^*)|} <
		\frac{\lmr(f_{\lt,\ltau})-\lmr(f_{\gt,\ltau})}{\lmr(f_{\lt,\gtau})-
			\lmr(f_{\gt,\gtau})}
		< \frac{1+\epsilon}{|T'_{\gt,\ltau,\gtau}(z^*)|}
	\]
	for all $0\le\lt,\gt,\ltau,\gtau\le T$ with $0<\gt-\lt<\delta$ and 
	$0\le\gtau-\ltau<\delta$. Finally, we find by Lemma \ref{Lem:FunOneAbs} a $\delta^*<\delta$ 
	with $1-\epsilon<|T'_{\gt,\ltau,\gtau}(z^*)|<1+\epsilon$ thus we have
	\[
		\frac{1-\epsilon}{1+\epsilon} <
		\frac{\lmr(f_{\lt,\ltau})-\lmr(f_{\gt,\ltau})}{\lmr(f_{\lt,\gtau})-
			\lmr(f_{\gt,\gtau})}
		< \frac{1+\epsilon}{1-\epsilon}
	\]
	for all $0\le\lt,\gt,\ltau,\gtau\le T$ with $0<\gt-\lt<\delta^*$ and 
	$0\le\gtau-\ltau<\delta^*$, so the proof is complete.
\end{proof}
\begin{proof}[Proof of Theorem \ref{The:LamAlmEve}]
	\begin{selflist}
	\item Since there is no risk of confusion, we omit the index $k$.
		First of all we consider the term
		\[
			S(f,[\lt,\gt],Z):= \sum_{l=0}^{s-1} 
			\big[ \lmr(f_{t_{l+1},t_l}) - \lmr(f_{t_l,t_l}) \big],
		\]
		where $Z=\{t_0, t_1, \ldots, t_s\}$ with $t_0=\lt$ and $t_s=\gt$ denotes a partition
		of the interval $[\lt,\gt]$ with $0\le\lt<\gt\le T$. 
		We denote by $|Z|$ the norm of the partition $Z$, i.e.
		\[
			|Z| = \max_{j=1,\ldots,s} |t_{j}-t_{j-1}|.
		\]
		Furthermore we set $S(f,t,Z):=S(f,[0,t],Z)$.
	\item We will show now, that $S(f,t,Z)$ converges to a value $c(t)\ge 0$ if 
		$|Z|\rightarrow 0$.
		Therefore, we consider two partitions $Z_1=\{t_0^*,\ldots,t_{s_1}^*\}$ and $Z_2$ of 
		the interval $[0,t]$ with $|Z_1|,|Z_2|<\delta$ where $\delta>0$. 
		Denote by $Z=\{t_0,\ldots,t_s\}$ the union of $Z_1$ and $Z_2$.
		By adding zeros we achieve
		\begin{multline*}
			|S(f,t,Z) - S(f,t,Z_1)| \le\\ 
			\sum_{l=0}^{s-1} |[\lmr(f_{t_{l+1},t_l})-\lmr(f_{t_l,t_l})]-
			[\lmr(f_{t_{l+1},\phi(t_l)})-\lmr(f_{t_l,\phi(t_l)})]|,
		\end{multline*}
		where $\phi(t_l):= t^*_{j}$ if $t_l\in [t^*_{j},t^*_{j+1})$ with $l=0,\ldots s-1$ and $j=0,
		\ldots,s_1 -1$. Consequently $|\phi(t_l)-t_l|\le \delta$. Thus we get
		\begin{multline*}
			|S(f,t,Z) - S(f,t,Z_1)| \le\\ 
			\sum_{l=0}^{s-1} |\lmr(f_{t_{l+1},t_l})-\lmr(f_{t_l,t_l})|\cdot
				\left |1- \frac{\lmr(f_{t_{l+1},\phi(t_l)})-\lmr(f_{t_l,\phi(t_l)})}
				{\lmr(f_{t_{l+1},t_l})-\lmr(f_{t_l,t_l})}\right|.
		\end{multline*}
		For any given $\epsilon>0$ we can choose $\delta>0$ (by using Proposition 
		\ref{Pro:DifQuoAbs}) in such a way that 
		\[
			1-\epsilon<\frac{\lmr(f_{t_l,\phi(t_l)})-\lmr(f_{t_{l+1},\phi(t_l)})}
				{\lmr(f_{t_l,t_l})-\lmr(f_{t_{l+1},t_l})} < 1+\epsilon
		\]
		holds for all $l=0,\ldots,s$. Thus, by using Proposition \ref{Pro:Monfunf}, we get
		\begin{align*}
			|S(f,t,Z) - S(f,t,Z_1)|&\le
				\epsilon\sum_{l=0}^{s-1} \big(\lmr(f_{t_{l+1},t_l})-\lmr(f_{t_l,t_l})
					\big)\\[-0.5\baselineskip]
				&\hspace{-4.1cm}<\epsilon\sum_{l=0}^{s-1} \big(\lmr(f_{t_{l+1},t_{l+1}})-
					\lmr(f_{t_{l+1},t_l})\big) +  
					\epsilon\sum_{l=0}^{s-1} \big(\lmr(f_{t_{l+1},t_l})-\lmr(f_{t_l,t_l})\big)\\
				&\hspace{-4.1cm}=\epsilon \cdot \big(\lmr(f_{t_s,t_s})-\lmr(f_{t_0,t_0})\big)	
					=\epsilon\cdot\big(\lmr(g_t)-\lmr(g_0)\big)
		\end{align*}
		By replacing $Z_1$ with $Z_2$ we get $|S(f,t,Z) - S(f,t,Z_2)|
		\le\epsilon\big(\lmr(g_t)-\lmr(g_0)\big)$. Consequently we have
		$|S(f,t,Z_1) - S(f,t,Z_2)|\le \epsilon\big(\lmr(g_t)-\lmr(g_0)\big)$,
		so $S(f,t,Z)$ converges to a value $c(t)\in[0,\infty)$ if $|Z|\rightarrow 0$.
	\item Next we will show that the function $t\mapsto c(t)$ is increasing. 
		Let $0< t_1<t_2\le T$, $Z_1(n):=\{0,\frac{t_1}{n},\frac{2t_1}{n},\ldots,t_1\}$,
		$Z_2(n):=\{t_1,t_1+\frac{t_2-t_1}{n},t_1+2\frac{t_2-t_1}{n},\ldots,t_2\}$ and
		$Z(n):=Z_1(n)\cup Z_2(n)$.
		Thus we have
		\begin{align} \tag{$\star\star$} \label{Equ:LipCon}
		\begin{split}
			c(t_2) - c(t_1) &=\lim_{n\rightarrow\infty} S(f,t_1,Z(n)) - S(f,t_2,Z_1(n))\\
			&=\lim_{n\rightarrow\infty} S(f,[t_1,t_2],Z_2(n)) \ge 0.
		\end{split}
		\end{align}
		Consequently, $t\mapsto c(t)$ is an increasing real-valued function. Thus this function
		is differentiable on $[0,T]\setminus \mathcal{N}_1$, where $\mathcal{N}_1$ denotes 
		a zero set with respect to the Lebesgue measure.
	\item Now we will show that $\lambda(t)$ exists almost everywhere.
		First we consider the function $t\mapsto\lmr(f_{t,t})=\lmr(g_t)$ which is strictly 
		increasing. Thus we find a zero set $\mathcal{N}_2$ such that 
		$t\mapsto\lmr(f_{t,t})=\lmr(g_t)$ is differentiable on $[0,T]\setminus \mathcal{N}_2$. 
		Let $t_0\in [0,T]\setminus \mathcal{N}$, where 
		$\mathcal{N}:=\mathcal{N}_1\cup \mathcal{N}_2$ and denote by ${(t_n)}_{n\in\N}$ a 
		sequence that converges to $t_0$ with $t_n> t_0$. Let $\epsilon>0$ 
		and $n_0\in\N$ be chosen in a way that $|t_n-t_0|< \delta$ holds for all $n\ge n_0$, 
		where $\delta$ is chosen according to Proposition \ref{Pro:DifQuoAbs} with respect to 
		$\epsilon$. For now we fix $n\ge n_0$.
		Let $Z=\{t_0,\ldots, t_s\}$ be an arbitrary partition of the interval 
		$[t_0,t_n]$. Thus we get
		\begin{align*}
			 &\left|\sum_{l=0}^{s-1} [\lmr(f_{t_{l+1},t_l})-\lmr(f_{t_l,t_l})] - 
			 [\lmr(f_{t_n,t_0})-\lmr(f_{t_0,t_0})]\right|\\
			 =& \left|\sum_{l=0}^{s-1} \big([\lmr(f_{t_{l+1},t_l})-\lmr(f_{t_l,t_l})] - 
			 [\lmr(f_{t_{l+1},t_0})-\lmr(f_{t_l,t_0})] \big)\right| = *
		\end{align*}
		where the first equality follows by adding zeros. By using Proposition \ref{Pro:DifQuoAbs} 
		and Proposition \ref{Pro:Monfunf} in combination with a telescoping sum as before, we see
		\begin{align*}
			* &\le \sum_{l=0}^s |\lmr(f_{t_{l+1},t_l})-\lmr(f_{t_l,t_l})|
			\cdot\left|1-\frac{\lmr(f_{t_{l+1},t_0})-\lmr(f_{t_l,t_0})}
			{\lmr(f_{t_{l+1},t_l})-\lmr(f_{t_l,t_l})}  \right|\\
			&< \sum_{l=0}^s \big(\lmr(f_{t_{l+1},t_l})-\lmr(f_{t_l,t_l})\big)\cdot \epsilon<
			\epsilon\cdot \big(\lmr(g_{t_n})-\lmr(g_{t_0})\big).
		\end{align*}
		Thus we get 
		\[
			|c(t_n)-c(t_0) - [\lmr(f_{t_n,t_0})-\lmr(f_{t_0,t_0})]|<
			\epsilon |\lmr(g_{t_n})-\lmr(g_{t_0})|,
		\]
		as $|Z|$ tends to zero. By dividing with $t_n-t_0$ we obtain
		\[
			\left|\frac{c(t_n)-c(t_0)}{t_n-t_0} - \frac{\lmr(f_{t_n,t_0})-
			\lmr(f_{t_0,t_0})}{t_n-t_0} \right|<
			\epsilon \left|\frac{\lmr(g_{t_n})-\lmr(g_{t_0})}{t_n-t_0} \right|.
		\]
		Since $t\mapsto c(t)$ and $t\mapsto \lmr(g_t)$ are differentiable at $t_0$, we get 
		finally the existence of $\lambda(t_0)$. Furthermore, we conclude 
		$\lambda(t_0)=\dot c(t_0)$.
		The other case $t_n\nearrow t_0$ follows in the same way.
	\end{selflist}
\end{proof}
However, this proof shows further interesting details, which we put together in the 
following proposition.
\begin{proposition} \label{Pro:ToolConstCoeff}
	Denote by $g_t$ and $f_{k;t,\tau}$ the corresponding mapping functions as mentioned before. 
	Then the limits 
	\[
		c_k(t):=\lim_{|Z|\rightarrow 0} S(f_k,t,Z)
	\]
	exists and form increasing functions
	$c_k:[0,T]\rightarrow[0,\infty)$ with $c_k(0)=0$ for all $k=1,\ldots,m$. 
	If $t\mapsto c_k(t)$ and $t\mapsto \lmr(g_t)$ are differentiable at $t_0$ then the
	relation $\lambda_k(t_0) = \dot c_k(t_0)$ holds, i.e. $\lambda_k$ exists in $t_0$.
	Furthermore, if $g'_t(0)=e^t$, i.e. $\lmr(g_t)=t$ for all $t\in[0,T]$ the functions 
	$t\mapsto c_k(t)$ are Lipschitz continuous in [0,T] for all $k=1,\ldots,m$
	and fulfil the condition $\sum_{k=1}^m c_k(t) = t$ for all $t\in[0,T]$.
\end{proposition}
\begin{proof}
	First of all we will prove the Lipschitz continuity of $c_k(t)$ if 
	$g_t=e^t$ holds for all $t\in[0,T]$ and $k=1,\ldots,m$. Therefore, let $k\in\{1,\ldots,m\}$
	and $0\le t_1< t_2\le T$ be fix. By using the same notation as in equation (\ref{Equ:LipCon})
	we get
	\begin{align*}
		c_k(t_2) - c_k(t_1) &=\lim_{n\rightarrow\infty} |S(f_k,[t_1,t_2],Z_2(n))| \\
			&\le \lim_{n\rightarrow\infty} \lmr(g_{t_2})- \lmr(g_{t_1}) = t_2 - t_1,
	\end{align*}
	so $c_k$ is Lipschitz continuous, since all values are nonnegative.
	Consequently the function $c(t):=\sum_{k=1}c_k(t)$ is Lipschitz continuous too 
	and we get by Theorem \ref{The:LoKCSDMul} $\dot c(t)= \sum_{k=1}^m \lambda_k(t)\equiv 1$ 
	almost everywhere in $[0,T]$. Thus $c(t)= t$ as $c(0)=0$.
\end{proof}
\begin{proof}[Proof of Corollary \ref{Cor:RegSimp}]
Let $t_0\in[0,T]$ and $k\in\{1,\ldots, m\}$ be fixed.
\begin{selflist}
	\item First of all we will prove that the value $\lambda_k(t_0)$ exists. 
	Since $k$ and $t_0$ are arbitrary chosen, the differential equation immediately follows by
	Corollary \ref{Cor:SimpleConCase}.
	
	Therefore let be $\gamma_1,\ldots,\gamma_m \in C^2([0,T])$, $k\in\{1,\ldots, m\}$ and 
	$t_0\in[0,T]$. 
	We denote $H:=f_{k;0,t_0}$. 
	Consequently $\delta:=H\circ \gamma_k \in C^2([0,T])$. 
	Next we define $G_t:=\D\setminus \delta[0,t]$ and denote by $h_t$ the unique conformal
	mapping that maps $G_t$ onto $\D$ with the normalization $h_t(0)=0$ and $h'(0)>0$.
	Observe that $H$ and $h_t$ depend on $t_0$.
	\begin{center}
		\includegraphics[scale=\scalefactor]{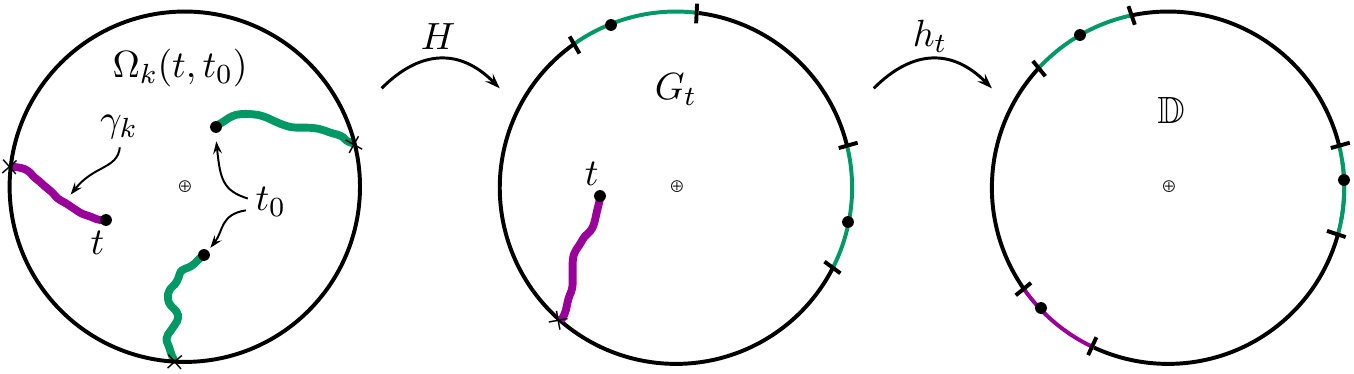}
	\end{center}
	Consequently, by applying Theorem 2 of \cite{EarleEpstein}, we find that the function 
	$t\mapsto\lmr(h_t)=\ln h'_t(0)$ is in $C^1[0,T]$, so that we have due to
	$f_{k;t,t_0} = h_t \circ H$:
	\[
		\frac{\text{d}}{\text{d}t} \lmr(h_t)\Big|_{t=t_0} = \lim_{t\rightarrow t_0}
		\frac{\lmr(f_{k;t,t_0})-\lmr(f_{k;t_0,t_0})}{t-t_0} = \lambda_k(t_0).
	\]
	\item Next we prove the continuity of $\lambda_k$. Therefore, let $\epsilon >0$ and denote
	$l(t):= \frac{\text{d}}{\text{d}t} \lmr(h_t)$. 
	By Proposition \ref{Pro:DifQuoAbs} we find a $\delta>0$ according to $\epsilon$.
	Since $l(t)$ is continuous, we find a $0<\rho<\delta$ so that 
	\[
		|l(t_0)-l(t)| < \epsilon
	\]
	holds for all $t\in (t_0-\rho,t_0+\rho)$.\\
	Moreover let  $M:=\max_{t\in[0,T]} l(t)$ and $t_1\in (t_0-\rho,t_0+\rho)$ be fixed.\\
	Let $\rho>0$ be chosen so small that
	\begin{align*}
		\Big|\frac{\lmr(f_{k;t_j+h,t_0})-\lmr(f_{k;t_j,t_0})}{h}-l(t_j)\Big|<\epsilon,\\
		\Big|\frac{\lmr(f_{k;t_j+h,t_j})-\lmr(f_{k;t_j,t_j})}{h}-\lambda(t_j)\Big|
			<\epsilon
	\end{align*}
	holds for $j=0,1$ and $|h|<\rho$. It follows
	\begin{align*}
		&\Big|\frac{\lmr(f_{k;t_1+h,t_0})-\lmr(f_{k;t_1,t_0})}{h} - 
		\frac{\lmr(f_{k;t_1+h,t_1})-\lmr(f_{k;t_1,t_1})}{h}\Big|\\
		&\hspace{1cm}= \frac{|\lmr(f_{k;t_1+h,t_0})-\lmr(f_{k;t_1,t_0})|}{h} \;
		\Big|1 - \frac{\lmr(f_{k;t_1+h,t_1})-\lmr(f_{k;t_1,t_1})}{\lmr(f_{k;t_1+h,t_0})
			-\lmr(f_{k;t_1,t_0})}\Big|\\
		&\hspace{1cm}\le (l(t_1)+\epsilon) \epsilon 
			\le (M+\epsilon ) \epsilon.
	\end{align*}
	Summarizing we have for all $t_1\in  (t_0-\rho,t_0+\rho)$
	\begin{align*}
		|\lambda_k(t_0)-\lambda_k(t_1)| &= |l(t_0)-\lambda_k(t_1)|\\
		&\le \Big|l(t_0)-\frac{\lmr(f_{k;t_1+h,t_1})-\lmr(f_{k;t_1,t_1})}{h}\Big| + \epsilon\\
		&\le \Big|l(t_0)-\frac{\lmr(f_{k;t_1+h,t_0})-\lmr(f_{k;t_1,t_0})}{h}\Big| + 
		\epsilon + (M+\epsilon) \epsilon \\
		&\le |l(t_0)-l(t_1)| + 2\epsilon + (M+\epsilon) \epsilon
		< 3\epsilon+ (M+ \epsilon) \epsilon,
	\end{align*}
	so the proof is complete.
\end{selflist}
\end{proof}

%% file: Input/Chapter5.tex
\section{Generalization to arbitrary domains}\label{Sec:GenArb}
Theorem \ref{The:LoKCSDMul} can also be extended to arbitrary domains. Let $\Omega^*$ be an 
arbitrary $n$ connected domain and $\gamma^*_k: [0,T]\rightarrow \bar \Omega^*\setminus\{z_0\}$ 
with $\gamma_k^*\big((0,T]\big)\subset \Omega^*\setminus\{z_0\}$ and 
accessible boundary points $\gamma_k^*(0)\in\partial \Omega^*$ for all $k=1,\ldots, m$. 
Denote by
\[
	g^*_t:\, \Omega^*(t) := \Omega^*\setminus\left (\bigcup_{k=1}^m \gamma_k^*\big((0,t] \big) 
	\right) \rightarrow D^*(t)
\]
a conformal mapping with $g^*_t(z_0)=0$ and ${g^*_t}'(z_0)>0$, where $D^*(t)$ stands for a 
circularly slit domain. As mentioned before, this function is unique. The function $f^*_{k,t,\tau}$ 
is defined analogous as before.
By $\xi^*_k(t)$ we mean the image of $\gamma_k^*(t)$ under the function $g^*_t$ for all
$k=1,\ldots, m$. Similar to Theorem \ref{The:LoKCSDMul} we get the following theorem.
\begin{theorem} 
	Denote by $g^*_t$ the corresponding mapping function as mentioned before. 
	Assume the values $\lambda^*_k(t_0)$ ($k=1,\ldots,m$) 
	exist for $t_0\in[0,T]$, then the function $g^*_t(z)$ is differentiable (w.r.t. $t$) 
	at $t_0$ for every $z\in\Omega^*(t_0)$ and it holds
	\[
		\dot g^*_{t_0}{(z)} = g^*_{t_0}(z) \sum_{k=1}^m \lambda_k^*(t_0) \cdot \Phi(\xi^*_k(t_0),
		g^*_{t_0}(z);t_0).
	\]
	Furthermore, the functions $\lambda^*_k(t_0)$ fulfill the condition 
	$\sum_{k=1}^m \lambda^*_k(t_0)= 1$, if the condition ${g^*_t}'(0)=c\,e^t$ holds in some
	neighborhood of $t_0$ with some $c>0$.
\end{theorem}
\begin{proof}
	Since there is a conformal mapping $h$ that maps $\Omega^*$ onto a circularly slit disk $\Omega$
	with $h(z_0)=0$ and $h'(z_0)>0$ we can transform the domain $\Omega^*(t)$ to $\Omega(t)$ by 
	using the mapping $h$. Thus we get 
	\begin{align*}
		\lambda_k^*(t_0)&= \lim_{t\rightarrow t_0} \frac{\lmr(f^*_{k;t,t_0})-\lmr(f^*_{k;t,t_0})}
			{t-t_0}
		=\lim_{t\rightarrow t_0} \frac{\lmr(f_{k;t,t_0}\circ h)-\lmr(f_{k;t,t_0}\circ h)}
			{t-t_0}\\
		&=\lim_{t\rightarrow t_0} \frac{\lmr(f_{k;t,t_0})-\lmr(f_{k;t,t_0})}{t-t_0}
		=\lambda_k(t_0)
	\end{align*}
	where $f$ corresponds to $\Omega(t)$ as before. 
	We can easily derive this result now from Theorem \ref{The:LoKCSDMul}.
\end{proof}

%% file: Input/Appendix.tex
\section*{Appendix} 
\begin{lemma} \label{Lem:ConvPhi}
	Let $(\xi_n)_{n\in\N}\subset\partial\D$ and $(t_n)_{n\in\N}\subset[0,T]$ be sequences 
	that converge to $\xi_0\in \partial\D$ and $t_0\in [0,T]$ respectively.
	Then the sequence $\big(\Phi(\xi_n,g_{t_n}(w),t_n)\big)_{n\in\N}$ converges
	locally uniformly to $\Phi(\xi_0,g_{t_0}(w),t_0)$ on $(\Omega_{t_n})_{n\in\N}$.
\end{lemma}
\begin{proof}
	Since $t_n\rightarrow t_0$ we have $\Omega_{t_n}\rightarrow\Omega_{t_0}$ in the sense 
	of kernel convergence due to Caratheodory. 
	
	We will prove that the sequence $h_n(z):=\Phi(\xi_n,g_{t_n}(z),t_n)$ is normal.
	Let $(h_{n_j})_{j\in\N}$ be a subsequence of $(h_{n})_{n\in\N}$.
	We show that there is a subsequence $(m_j)_{j\in\N}$ of $(n_j)_{j\in\N}$ so that 
	for all compact sets $K\subset \Omega_{t_0}$ and all $\epsilon>0$ exists 
	a $j_0\in\N$, so that $|h_{m_j}(z)-h_0(z)|<\epsilon$ holds for all $z\in K$ and
	all $j\ge j_0$.
	
	Therefore let $(K_n)_{n\in\N}$ be a sequence of compact sets, so that
	$K_n\subset \Omega_{t_n}$, $K_n\subset \mathring K_{n+1}$ and 
	$\bigcup_{n\in\N} K_n = \Omega_{t_0}$. By Montel's theorem we find a 
	subsequence $(m_{1,j})_{j\in\N}$ of $(n_j)_{j\in\N}$, so that
	$h_{m_{1,j}}$ converges uniformly on $K_1$. Inductively we find for all $n\in\N$ 
	a subsequence $(m_{n+1,j})_{j\in\N}$ of $(m_{n,j})_{j\in\N}$ so that 
	$h_{m_{n+1,j}}$ converges uniformly on $K_n$.
	Consequently we set $m_j:=m_{j,j}$, so $h_{m_j}$ converges locally uniformly on
	$(\Omega_{t_{m_j}})_{j\in\N}$ to the analytic function $h_0$.
		
	Next, we are going to show $h_0\equiv \Phi(\xi_0,g_{t_0},t_0)$. Since $h_0$ is not constant,
	$h_0$ is a conformal mapping of $\Omega_{t_0}$ onto the $n$-connected domain $R$.
	Hereby $R$ is necessarily the right halfplane minus slits parallel to the imaginary
	axis. By using the uniqueness of the function $\Phi$ we find
	$h_0\equiv \Phi(\xi_0,g_{t_0},t_0)$.

	As all subsequences $(h_{m_j})_{j\in\N}$ converge to the same function $h_0$, also the whole
	sequence $(h_n)_{n\in\N}$ converges locally uniformly to $h_0$ on $(\Omega_{t_n})_{n\in\N}$.
\end{proof}